\documentclass[11pt,a4paper]{article}
\usepackage[utf8]{inputenc}
\usepackage[english]{babel}
\usepackage{amsmath}
\usepackage{amsfonts}
\usepackage{amssymb}
\usepackage[left=3cm,right=3cm,top=3cm,bottom=3cm]{geometry} 

\usepackage{amssymb}
\usepackage{amsthm}

\usepackage{amsfonts}
\usepackage{amssymb}
\usepackage{indentfirst}
\usepackage{color,colortbl}
\usepackage[table]{xcolor}
\usepackage{stmaryrd}

\usepackage{array}

\usepackage{float}

\usepackage[pdftex]{graphicx}
\usepackage{caption}

\usepackage{mathrsfs}
\usepackage{tikz}
\newtheorem{theorem}{Theorem}[section]

\newtheorem{prop}[theorem]{Proposition}

\newenvironment{rmq}[1][Remark]{\begin{trivlist}
\item[\hskip \labelsep {\bfseries #1}]}{\end{trivlist}}

\usepackage{natbib}

\usepackage[titletoc]{appendix} 

\usepackage{hyperref} 
\hypersetup{
colorlinks=true, 
breaklinks=true, 
urlcolor= blue, 
linkcolor= blue, 
citecolor=blue, 
} 

\title{\textbf{Partial Least Square\\ A new statistical insight through the prism of orthogonal polynomials}}

\author{M\'elanie~Blaz\`ere,
        Jean-Michel~Loubes
        and~Fabrice~Gamboa}

\date{} 

\begin{document}

\maketitle

\begin{abstract}
Partial Least Square (PLS) is a dimension reduction method used to remove multicollinearities in a regression model. However contrary to Principal Components Analysis (PCA) the PLS components are also choosen to be optimal for predicting the response $Y$. In this paper we provide a new and explicit formula for the residuals. We show that the residuals are completely determined by the spectrum of the design matrix and by the noise on the observations. Because few are known on the behaviour of the PLS components we also investigate their statistical properties in a regression context. New results on regression and prediction error for PLS are stated under the assumption of a low variance of the noise.
\end{abstract}

\begin{center}
{\bf \small Keywords} 
\end{center}
\begin{center}
Partial Least Square, multivariate regression, multicollinearity, dimension reduction, constrainsted least square, orthogonal polynomials, prediction error.
\end{center}

\section{Introduction}
Partial Least Square (PLS), introduced in 1985 by \cite{WO85}, is nowadays a widely used dimension reduction technique in multivariate analysis especially when we have to handle high dimensional or highly correlated data in a regression context. Originally designed to remove the problem of multicollinearity in the set of explanatory variables, PLS acts as a dimension reduction method by creating a new subset of variables which are also optimal for predicting the output variable. During the last decades this method has been developped and studied for a large part by \cite{HE88} and by \cite{FRANCK93}. Partial Least Square was originally developped for chemometrics applications (see for example \cite{WOLD01PLS} and \cite{FRANCK93}) but gained attention in biosciences in particular in the analysis of high dimensional genomic data. We refer for instance to \cite{BOU07} or to  \cite{MR2457048}  for various applications in this field. 

If the PLS method proved helpful in a large variety of situations, this iterative procedure is complex and  little is known about its theoretical properties but PLS has been well investigated by pratical experiments. To name just a few, \cite{NAES85} discussed theoretical and computational considerations of PLS and PCR (Principal Component Regression) on simulated and real data. \cite{FRANCK93} provided a heuristic comparison of the performances of OLS, PCR, Ridge regression and PLS in different situations. \cite{GAR94} compared PLS with four other methods (ordinary least squares, forward variable selection, principal components regression, and a Stein shrinkage method) through simulations. Only  very recently, some theoretical insights have been given by \cite{DEL12} for functionnal data. 
 
In this work, we provide a new direction to analyze some statistical aspects of the PLS method. For this, we will draw connections between PLS, Krylov subspaces and the regularization of inverse problems (see \cite{MR1408680}).\\
\indent The paper falls into the following sections. In Section \ref{section:framework} we present the framework within which we study PLS and we briefly recall what is the PLS method. We also highlight the connection between PLS and Krylov subspaces. Because the directions of the new subspace onto which we project the observations depend on the response variable it is quite difficult to study the statistical properties of PLS using just the algorithmic construction of the new subspace.
  In this paper we adopt the point of view of PLS viewed as a constrainsted least square problem and use its connection with inverse problem with a statistical point of view. In Section \ref{section:Krylov} we highlights the connection between PLS and the minimization of the least squares over polynomial subspaces. Then in Section \ref{section:residual} we provide a new formulation of the residuals for each direction defined by the eigenvectors of the covariance matrix. The interest of such a formulation rests on the fact that it provides an explicit expression of the residuals in terms of both the noise on the observations and on the eigenelements of the covariance matrix. This formulation will enable a further study of the PLS method performance in a regression framework. In Section \ref{section:error} we study PLS in the context of a high-dimensional multiple regression model. We first define the model under study. Then we detail our main results for noisy sample and new statistical aspects of PLS. We provide bounds for the empirical risk and for the mean square error of prediction under the assumption of a low variance of the noise. Asymptotic properties of the prediction error are also discussed. We also highlight the limitations of this method according to the features of the data.

\section{Presentation of the framework}
\label{section:framework}
\subsection{Notation and remarks}
We first introduce some of the notation we use in this paper. 
By $\left\langle x,y\right\rangle $ we denote the inner product between the vectors
$x, y \in \mathbb{R}^{n}$. The transpose of a matrix $A$ is denoted by $A^{T}$ and it depends on the underlying
inner product, i.e. $\left\langle Ax,y \right\rangle=\left\langle x,A^{T}y \right\rangle$. The induced vector norm is $\Vert x\Vert = \sqrt{\left\langle x,x\right\rangle }$. In most cases we work with the Euclidean inner product i.e. $\left\langle x,y\right\rangle=x^{T}y $ and
the induced norm is the $\ell_{2}$-norm. For any positive definite matrix $M$, the $M$-inner product is defined as $\left\langle x,y \right\rangle_{M} = x^{T}My$ and the operator norm is given by $\Vert M\Vert = \underset{\Vert x\Vert=1}{\rm{max}}\Vert Mx\Vert$. We simply denote by $I$ the identity matrix with the corresponding dimension. 
For every $k \in \mathbb{N}$ we denote by $\mathcal{P}_{k}$ the set of the polynomials of degree less than $k$ and by $\mathcal{P}_{k,1}$ the set of the polynomial of degree less than $k$ whose constant term equals 1.

The figures which appear in this paper are the result of simulations which have been performed with R using the package \texttt{plsgenomics} developped by Boulesteix and al. The function \texttt{pls.regression} has been used to fit the model. 

\subsection{The regression model}
We consider the following regression model
\begin{equation}
\label{eq:regression-model}
Y=X\beta^{*}+\varepsilon
\end{equation}
 where 
\begin{itemize}
\item [\textbullet]$Y=(Y_{1},...,Y_{n})^{T} \in \mathbb{R}^{n}$ is the vector of the observed outcome, also called the response. 
\item [\textbullet] $X=(X_{ij})_{1\leqslant i\leqslant n,1\leqslant j\leqslant p} \in \mathbb{M}_{n\times p}$ is the design matrix which is considered as fixed and contains the predictors.
\item [\textbullet]$\beta^{*}=(\beta_{1}^*,...,\beta_{p}^*)^{T} \in \mathbb{R}^{p} $ is the unknown parameter vector and represents the variables of interest. 
\item[\textbullet]$\varepsilon=(\varepsilon_{1},...,\varepsilon_{n})^{T} \in \mathbb{R}^{n}$ captures the noise. 
\end{itemize}
In other word we are concerned with finding a good approximation $\hat{\beta}$ of the solution $\beta^{*}$ of the above linear problem where only noisy observations are available.
For the moment we only assume that the real variables $\varepsilon_{1},...,\varepsilon_{n}$ are unobservable i.i.d random variables.
We allow $p$ to be much larger than $n$ i.e $p\gg n$. We denote by $r$ the rank of $X^{T}X$. Of course $r\leqslant \textrm{min}(n,p)$. The aim is to estimate the unknown parameter $\beta^*$ from the observations of the pairs $(Y_{i},X_{i})_{1\leq i \leq n}$. The usual ordinary least squares (OLS) estimates $\beta^{*}$ by $\hat{\beta }$ where $$\hat{\beta }\in \underset{\beta \in \mathbb{R}^{p}}{\textrm{argmin}}\Vert Y-X\beta\Vert^{2}.$$ However we know that in the case of highly correlated explanatory variables and/or when the explanatory variables outnumber the observations i.e $p\gg n$ the regression model is ill conditioned and the OLS estimator behaves badly. The estimated parameter can be very unstable and far from the target leading to unaccurate predictions. To remove the problem of multicollinearities in regression model a solution consists of creating latent variables using Principal Component Analysis (PCA). However, the new variables are chosen to explain $X$ but they may not explain $Y$ well. \cite{JO82} provided several real-life examples where the principal components corresponding to small eigenvalues have high correlation with $Y$. To avoid this problem a possible solution is to use Partial Least Square which has been heavily promoted as an alternative to OLS in the literature.

\subsection{The PLS method}
In this subsection we briefly recall the method. The PLS method, introduced by \cite{WO85}, emerged in order to remove the problem of multicollinearity in a regression model (when the number of covariates is large or when there are dependancies between them). In fact PLS is a statistical method whose challenge is to find principal components that explain $X$ as well as possible and are also good predictors for $Y$.

The PLS method at step $K$ (where $K\leqslant r$) consists in finding $(w_{k})_{1\leq k\leq K}$ and $(t_{k})_{1\leq k\leq K}$ which maximize $\left[ \textrm{Cov}(Y,Xw_{k})\right] ^{2}$ under the constraint
\begin{enumerate}
\item $\Vert w_{k}\Vert^{2}=w_{k}^{T}w_{k}=1$
\item $t_{k}=Xw_{k}$ is orthogonal to $t_{1},...,t_{k-1}$ i.e $w_{k}X^{T}Xw_{l}=0$ for $l=1,...,k-1$.
\end{enumerate} 
Therefore the PLS method is a procedure which iteratively constructs a subspace of dimension $K$ (spanned by $(w_{k})_{1\leq k \leq K}$) in such a way that the new latent variables $(t_{k})_{1\leq k \leq K}$ (which are the projections of the original ones) maximize both the correlation with the response and the variance of the explanatory variables. The original algorithms were developped by \cite{WO83} and a decade later by \cite{MA92}. 

Once the latent variables $(t_{k})_{1\leq k \leq K}$ are built, one can compute the linear regression of $Y$ on $t_{1},...,t_{K}$ to estimate $\beta^{*}$. We can notice that this method is of particular interest because it can analyze data with strongly correlated, noisy and numerous covariates. Furthermore dimension reduction and regression are performed simultaneously. 
We refer to \cite{HE90,HE88,HE01} for the the study of the main properties of PLS and to \cite{KRA07} for a complete overview on the recent advances on PLS. Proposition \ref{prop:CLS} below recalls what appears to us as one of the main result on PLS because it is the starting point of our work. This proposition shows that the PLS estimator at step $K$ is defined as the argument which minimizes the least square over some particular subspace of dimension $K$.

\begin{prop}\cite{HE90}
\label{prop:CLS}
\begin{equation}
\hat{\beta}_{K}^{PLS}= \underset{\beta \in \mathcal{K}^{K}(X^{T}X, X^{T}Y)}{\textrm{argmin}}\Vert Y-X\beta\Vert^{2}
\label{eq:esti_betabis} 
\end{equation}
where $\mathcal{K}^{k}(X^{T}X,X^{T}Y)=\left\lbrace X^{T}Y, (X^{T}X)X^{T}Y,..., (X^{T}X)^{k-1}X^{T}Y\right\rbrace$.
\end{prop}
The space $\mathcal{K}^{k}(X^{T}X,X^{T}Y)$ spanned by $X^{T}Y, (X^{T}X)X^{T}Y,..., (X^{T}X)^{k-1}X^{T}Y$ (and denoted by $\mathcal{K}^{k}$ when there is no possible confusion) is called the $k^{th}$ Krylov subspace. We refer to \cite{SAAD92} for a further study of these spaces.
We can notice that, as for PCR, PLS is a constrainted least square estimator where the constraints are not on the norm of the parameter (as for Ridge regression or for the Lasso) but are linear constaints which ensure that the estimated parameter belongs to the Krylov subspace associated to  $X^{T}X$ and to $X^{T}Y$. However we have to be careful that contrary to PCR  the PLS linear constraints are random. 

Using this connection with Krylov subspaces \cite{PHATAK02} showed that the PLS iterates are the same as the ones of the Conjugate Gradient(CG). Thus PLS can also be viewed as CG applied in the statistical framework of linear regression models. Phatak and de Hoog also used the connection between CG, Lanczos method and PLS to give simplier proofs of two known results. The first one is the shrinkage properties of PLS  $( \Vert \hat{\beta}_{k}^{PLS}\Vert\leqslant \Vert \hat{\beta}_{k+1}^{PLS}\Vert) $ proved by \cite{JONG95} and the second is the fact that PLS fits better than PCR $( \Vert \hat{\beta}_{k}^{PLS}\Vert\leqslant\Vert \hat{\beta}_{OLS}\Vert) $ proved by \cite{GOU96}.

\section{A close connection between PLS and orthogonal polynomials}
\label{section:Krylov}
In this section we show that the PLS solution can be written as the polynomial solution of a minimization problem. Then we prove that the sequence of the residuals in each eigenvectors direction can be expressed as a sequence of orthogonal polynomials with respect to a discrete measure. This measure depends on the eigenvalues of the design matrix and on the projection of the response onto the associated eigenvectors.

\subsection{A useful tool to analyze the properties of PLS}
Consider the singular value decomposition of $X$ given by $$X=UDV^{T}$$ where
\begin{itemize}
\item $U^{T}U=I$ and $u_{1},...,u_{p}$ are the columns of $U$ and form an orthonormal basis of $\mathbb{R}^n$.
\item $V^{T}V=I$ and $v_{1},...,v_{p}$ are the columns of $V$ and form an orthonormal basis of $\mathbb{R}^n$.
\item $D\in \mathbb{M}_{n,p}$ is a matrix which contains $(\sqrt{\lambda_{1}},...,\sqrt{\lambda_{n}})$ on the diagonal and zero anywhere else.
\item We assume that $\lambda_{1}\geq \lambda_{2}\geq ....\geq\lambda_{n}>0=\lambda_{n+1}=...=\lambda_{p}$. In other words we assume that $X^{T}X$ is of full rank i.e of rank $n$.
\end{itemize}
Of course we have $X^{T}u_{i}=\sqrt{\lambda_{i}}v_{i}$, $XX^{T}u_{i}=\lambda_{i} u_{i}$ for all $i=1,...,n$ and $Xv_{i}=\sqrt{\lambda_{i}}u_{i}$, $X^{T}Xv_{i}=\lambda_{i} v_{i}$ for all $i=1,...,p$. We define $\tilde{\varepsilon}_{i}:= \varepsilon^{T}u_{i}$, $i=1,...,n$ and $\tilde{\beta}^{*}_{i}:={\beta^{*}}^{T}v_{i}$, $i=1,...,p$ the projections of $\varepsilon$ and $\beta^{*}$respectively onto the right and left eigenvectors of $X$.

We assume that $k\leqslant \mu+1$ where $\mu$ is the grade of $X^{T}Y$ with respect to $X^{T}X$ i.e the degree of the nonzero monic polynomial $P$ of lowest degree such that $P(X^{T}X)X^{T}Y=0$. The Krylov subspace $\mathcal{K}^{k}$ is of dimension $k$ if and only if $\mu\geqslant k-1$ in such a way that in this case we have $\textrm{dim}(\mathcal{K}^{k})=k$. 

We can notice that the maximal dimension of the Krylov subspace sequence is also linked to the number of non zero eigenvalues $\lambda_{i}$ for which $u_{i}^{T}Y \neq 0$ (see \cite{HE90}). These particular eigenvalues are called the releavant eigenvalues. If the number of releavant eigenvalues is $n$ then the maximal dimension of the Krylov subspaces sequence is also $n$ and for all $k\leq n$ the dimension of $\mathcal{K}^{k}$ is exactly $k$. In particular if $X^{T}Y=\sum_{l=1}^{k}\sqrt{\lambda_{i_{l}}}(u_{i_{l}}^{T}Y)v_{i_{l}}$ then the PLS iterations will terminate in at most $k$ iterations.
\subsection{Link with the regularization of inverse problems methods: a minimization problem over polynomials}
We recall that $\mathcal{P}_{k}$ is the set of the polynomials of degree less than $k$ and $\mathcal{P}_{k,1}$ the set of the polynomial of degree less than $k$ whose constant term equals one.
By combining formula (\ref{eq:esti_betabis}) which expresses the PLS estimator as a constrainted least square over Krylov subspace with the definition of $\mathcal{K}^{k}$ it is easy to show that the PLS estimator can also be expressed as the solution of a minimization problem over polynomials.

\begin{prop}
\label{prop:poly}
For $k\leq n$ we have
\begin{equation}
\label{eq:PLS_poly}
\hat{\beta}_{k}=\hat{P}_{k}(X^{T}X)X^{T}Y
\end{equation}
where $\hat{P}_{k}$ is a polynomial of degree less than $k-1$ which satisfies $$\Vert Y-X\hat{P}_{k}(X^{T}X)X^{T}Y\Vert^{2}=\underset{P \in \mathcal{P}_{k-1}}{\textrm{argmin}}\Vert Y-XP(X^{T}X)X^{T}Y\Vert^{2}$$
and 
\begin{equation}
\label{eq:residu-poly}\Vert Y-X\hat{\beta}_{k}\Vert^{2}=\Vert \hat{Q}_{k}(XX^{T})Y\Vert^{2}=\underset{Q\in \mathcal{P}_{k,1}}{\textrm{min}}\Vert  Q(XX^{T})Y\Vert^{2}
\end{equation}
where $\hat{Q}_{k}(t)=1-t\hat{P}_{k}(t)$ is a polynomial of degree less than $k$ and of constant term equals to one.
\end{prop}

Notice that for all $k\leq n$ we have $X\hat{\beta}_{k}=\hat{\Pi}_{k}Y$ where $\hat{\Pi}_{k}$ is the orthogonal projector onto the random space $\mathcal{K}^{k}(XX^{T}, XX^{T}Y)$ of dimension $k$. In particular for $k=n$, we have $X\hat{\beta}_{n}=Y$ and $\Vert Y-X\hat{\beta}_{n}\Vert^{2}=0$ because $\mathcal{K}^{n}(XX^{T}, XX^{T}Y)$ is of dimension $n$. In the following we will omit this trivial case.

Proposition \ref{prop:poly} shows that the PLS method is another regularization method for ill-posed inverse problem (see \cite{MR1408680}). In fact when the explanatory variables are highly correlated or when they outnumber the number of observations the regression model we consider is ill-posed. The idea behind PLS is to approximate the ill-posed problem by a family of nearby well-posed problem by seeking for regularization operator $\mathcal{R}_{\alpha}$ such that $\mathcal{R}_{\alpha}(X^{T}X)X^{T}Y$ close to $\beta^{*}$ where $\mathcal{R}_{\alpha}$ is under a polynomial form. In fact the polynomials $\hat{P}_{k}$ play the role of $\mathcal{R}_{\alpha}$ and $\alpha=k$ is the regularization parameter. We also refer to \cite{BLAN12} to take a more in depth look on statistical inverse problems and Conjugate Gradient because PLS is closely related to Conjugate Gradient with a statistical point of view.

The idea of considering the Krylov subspace and thus polynomial approximation is at the heart of the issue for PLS. We present below a result which gives a good reason to search for polynomial approximations. Indeed the theorem of Cayley-Hamilton tells us that we can represent the inverse of a nonsingular matrix $A$ in terms of the powers of $A$. It is no longer the case for a singular matrix because the inverse does not exist. But the idea behind PLS remains quite the same for non singular matrix. It consists of using Krylov subspaces to approximate the pseudo inverse as a polynomial in $A$. In fact according to (\ref{eq:PLS_poly}) the PLS estimator $\hat{\beta}_{k}$ is of the form $\hat{P}_{k}(X^{T}X)X^{T}Y$ where $\hat{P}_{k}$ is a polynomial of degree less than $k-1$ and thus consists in a kind of regularization of the inverse of $X^{T}X$. Notice that since $\textrm{dim}(\mathcal{K}^{k})=k$ the polynomial $\hat{P}_{k}$ is in fact of degree exactly $k-1$. If $X^{T}X$ is invertible then the PLS method generates a sequence of polynomial approximation of the inverse of $X^{T}X$ and when $k=n$ we recover the inverse of $X^{T}X$ exactly.

So PLS is also equivalent to finding an optimal polynomial $\hat{Q}_{k}$ of degree $k$ with $\hat{Q}_{k}(0)=1$ minimizing  $\Vert  Q(XX^{T})Y\Vert^{2}$. Notice that if there exists a polynomial $Q$ of degree $k$ with $Q(0)=1$ small on the spectrum of $XX^{T}$ then $\Vert Y-X\hat{\beta}_{k}\Vert^{2}$ will be small too. In particular if the eigenvalues are clustered into $k$ groups (i.e can be divided into $k$ groups whose diameter are very small) then $\Vert Y-X\hat{\beta}_{k}\Vert^{2}$ has a good chance to be small as well. The polynomial $\hat{Q}_{k}$ quantifies the quality of the approximation of the response $Y$ at the $k^{th}$ step. We call these polynomials the residuals. 

\subsection{Link with orthogonal polynomials}
\begin{flushleft}
In this subsection we first prove that the sequence of polynomials $\left(\hat{Q}_{k}\right)_{1\leqslant k <n} $ defined in  Proposition \ref{prop:poly}  is orthogonal with respect to a discrete measure denoted by $\hat{\mu}$. 
\end{flushleft}

\begin{prop}
$\hat{Q}_{1},\hat{Q}_{2},...,\hat{Q}_{n-1}$ is a sequence of orthonormal polynomials with respect to the measure 
$$d\hat{\mu}(\lambda)=\sum_{j=1}^{n}\lambda_{j}(u_{j}^{T}Y)^{2}\delta_{\lambda_{j}}.$$
\label{prop: poly-ortho}
\end{prop}

The support of the measure $\hat{\mu}$ consists of the $(\lambda_{i})_{1\leq i \leq n}$ and the weights depend on $(\lambda_{i})_{1\leq i \leq n}$ and $(u_{j}^{T}Y)_{1\leq i \leq n}$. These last quantities capture both the variation in $X$ and the correlation between $X$ and $Y$.

\section{A new expression for the residuals in the eigenvectors direction}
\label{section:residual}
 \subsection{Main Result}
If the PLS properties are not completely understood it is partly because the solution is a non linear function of the data $Y$. PLS is an iterative method and therefore if we perturb $Y$ the perturbation propagates through the sequence of Krylov subspaces in a non linear way which makes difficult the explicit study of the PLS estimator.
 In this section we provide a new explicit and exact formulation of the residuals which clearly shows how the disturbance on the observations impacts on the residuals.

In this section a new and exact expression for $\hat{Q}_{k}(\lambda_{i})$ is proposed for all $k=1,...,n-1$ and all $i=1,...,n$.

\begin{theorem}
\label{theo: expression-det-2}
Let $k\leq n$ and $$I_{k}^{+}=\left\lbrace n\geq j_{1}>...>j_{k}\geq 1\right\rbrace.$$

We have
\begin{equation}
\label{eq:final-expression-residuals}
\hat{Q}_{k}(\lambda_{i})=\sum_{(j_{1},..,j_{k})\in I^{+}_{k}}\left[ 
\dfrac{\hat{p}_{j_{1}}^{2}...\hat{p}_{j_{k}}^{2}\lambda_{j_{1}}^{2}...\lambda_{j_{k}}^{2}V(\lambda_{j_{1}},...,\lambda_{j_{k}})^{2}}{\sum_{(j_{1},..,j_{k})\in I^{+}_{k}} \hat{p}_{j_{1}}^{2}...\hat{p}_{j_{k}}^{2}\lambda_{j_{1}}^{2}...\lambda_{j_{k}}^{2}V(\lambda_{j_{1}},...,\lambda_{j_{k}})^{2}}\right] \prod_{l=1}^{k}(1-\frac{\lambda_{i}}{\lambda_{j_{l}}}).
\end{equation}
where $\hat{p}_{i}:= p_{i}+ \tilde{\varepsilon}_{i} $ with $p_{i}:=(X\beta^{*})^{T}u_{i}=\sqrt{\lambda_{i}}\tilde{\beta}^{*}_{i}$ and $\tilde{\varepsilon}_{i}:= \varepsilon^{T}u_{i}$.
\end{theorem}

For all $k < n$ we recover that $\hat{Q}_{k}$ is a polynomial of degree $k$ and $\hat{Q}_{k}(0)=1$.
The expression of $\hat{Q}_{k}(\lambda_{i})$ given in Proposition \ref{prop: expression-det} depends explicitly on the observations noise and on the eigenelements of $X$ contrary to the expression provided in the paper of \cite{LIN00}. Formula (\ref{eq:final-expression-residuals}) is also valid for $k=n$ but in this case we recover that $\hat{Q}_{n}(\lambda_{i})=0$ for all $i=1,...,n$.

 Now assume that there are only $k$ distinct eigenvalues among the $n$ ones and denote by $\tilde{\lambda}_{1}$,...,$\tilde{\lambda}_{k}$  the different representatives. Then for all $i=1,...n$ formula (\ref{eq:final-expression-residuals}) implies
$$\hat{Q}_{k}(\lambda_{i})=\prod_{j=1}^{k}(1-\frac{\lambda_{i}}{\tilde{\lambda}_{j}})=0. $$
Thus the residuals along each eigenvectors equal zero at step $k$ if there are less than $k$ different non zero eigenvalues (notice that this is of course the case when $k=n$). Furthermore if we assume that there exists only $k$ eigenvectors denoted by $u_{\overline{j_{1}}},...,u_{\overline{j_{k}}}$ such that $\hat{p}_{\overline{j_{1}}}\neq 0$,...,$\hat{p}_{\overline{j_{k}}}\neq 0$ then formula (\ref{eq:final-expression-residuals}) becomes 
$$\hat{Q}_{k}(\lambda_{i})=\prod_{j=1}^{k}(1-\frac{\lambda_{i}}{\lambda_{\overline{j_{l}}}}).$$
Therefore for all $\lambda \in \left\lbrace  \lambda_{\overline{j_{1}}},...,\lambda_{\overline{j_{k}}}\right\rbrace $, 
$\hat{Q}_{k}(\lambda)=0$
 and thus we find 
$\Vert Y-X\hat{\beta}_{k}\Vert^{2}=\Vert \hat{Q}_{k}(XX^{T})Y\Vert^{2}=\sum_{i=1}^{n}\hat{Q}_{k}(\lambda_{i})^{2}(u_{i}^{T}Y)^{2}= 0.$

For all $((j_{1},...,j_{k})) \in I_{k}^{+}$, let $$\hat{w}_{j_{1},..,j_{k}}:=\dfrac{\hat{p}_{j_{1}}^{2}...\hat{p}_{j_{k}}^{2}\lambda_{j_{1}}^{2}...\lambda_{j_{k}}^{2}V(\lambda_{j_{1}},...,\lambda_{j_{k}})^{2}}{\sum_{(j_{1},..,j_{k})\in I^{+}_{k}} \hat{p}_{j_{1}}^{2}...\hat{p}_{j_{k}}^{2}\lambda_{j_{1}}^{2}...\lambda_{j_{k}}^{2}V(\lambda_{j_{1}},...,\lambda_{j_{k}})^{2}}$$
and $$g_{j_{1},..,j_{k}}(x)=\prod_{l=1}^{k}(1-\frac{x}{\lambda_{j_{l}}} ).$$ 
Notice that this last function is again a polynomial in $x$ of degree $k$ whose constant term is equal to one and is zero at $\lambda_{j_{1}},...,\lambda_{j_{k}}$ which are elements in the spectrum of $XX^{T}$. We have $$\hat{Q}_{k}(\lambda_{i})=\sum_{(j_{1},..,j_{k})\in I^{+}_{k}} \hat{w}_{(j_{1},..,j_{k})}g_{j_{1},..,j_{k}}(\lambda_{i}).$$
Besides, for all $(j_{1},...,j_{k}) \in I_{k}^{+}$, $0<\hat{w}_{(j_{1},..,j_{k})}\leq 1$ and $\sum_{(j_{1},..,j_{k})\in I^{+}_{k}}\hat{w}_{(j_{1},..,j_{k})}=1$. Thus the weights $(\hat{w}_{(j_{1},..,j_{k})})_{I^{+}_{k}}$ are probabilities.
Therefore $\hat{Q}_{k}(\lambda_{i})$ is the sum over all elements in $I_{k}^{+}$ of $g_{j_{1},..,j_{k}}(\lambda_{i})$ weighted by the probabilities $\hat{w}_{(j_{1},..,j_{k})}$. It is a kind of barycenter of all the polynomials in $\mathcal{P}_{k,1}$ whose roots are subsets of $\left\lbrace \lambda_{1},...,\lambda_{n}\right\rbrace $. The weights are not easy to interpret but they are even greater when the magnitude and the distance between the involved eigenvalues are large and the contribution of the response along the associated eigenvectors is important.
From formula (\ref{eq:final-expression-residuals}) we can state in a very large way that $$\mid\hat{Q}_{k}(\lambda_{i})\mid\leq \underset{I_{k}^{+}}{\textrm{max}}\left( \prod_{l=1}^{k}\left\lvert 1-\frac{\lambda_{i}}{\lambda_{j_{l}}}\right \rvert\right). $$
In particular if $ \lambda_{1}(1-\varepsilon) \leq\lambda_{i} \leq \lambda_{n}(1+\varepsilon)$ then $$\mid\hat{Q}_{k}(\lambda_{i})\mid\leq \varepsilon^{k}.$$

Here is an example of the residuals path with respect to the eigenvectors directions for $100$ nonzero eigenvalues which are distributed around $10$ different values. We also represent the residuals only for the extremal eigenvalues to better see the difference of behaviour.
\begin{figure}[H]

  \begin{minipage}[b]{0.5\linewidth}
   \centering
       \includegraphics[scale=0.4]{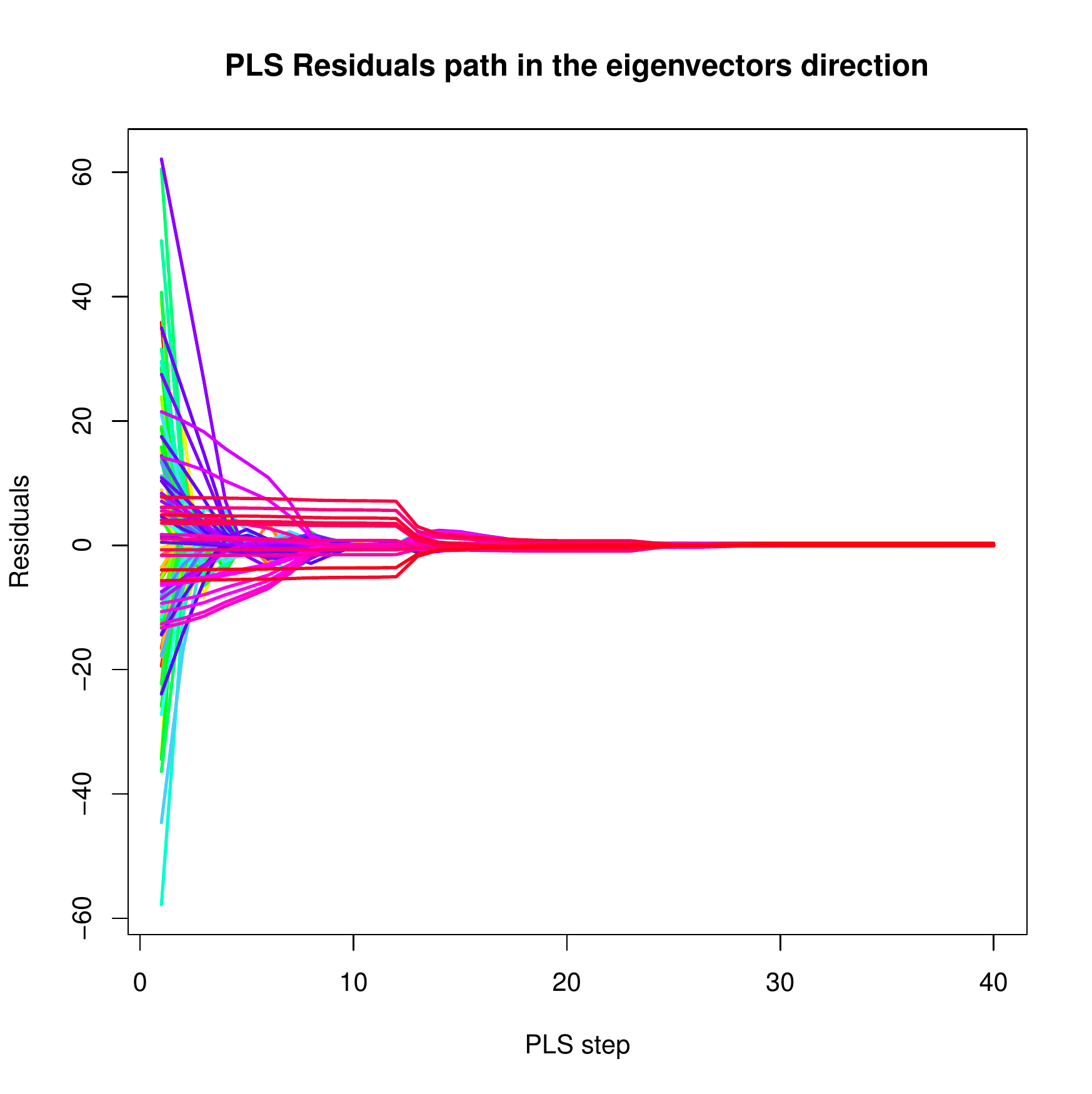} 
\caption{}
  \end{minipage}
\begin{minipage}[b]{0.5\linewidth}
   \centering
   \includegraphics[scale=0.4]{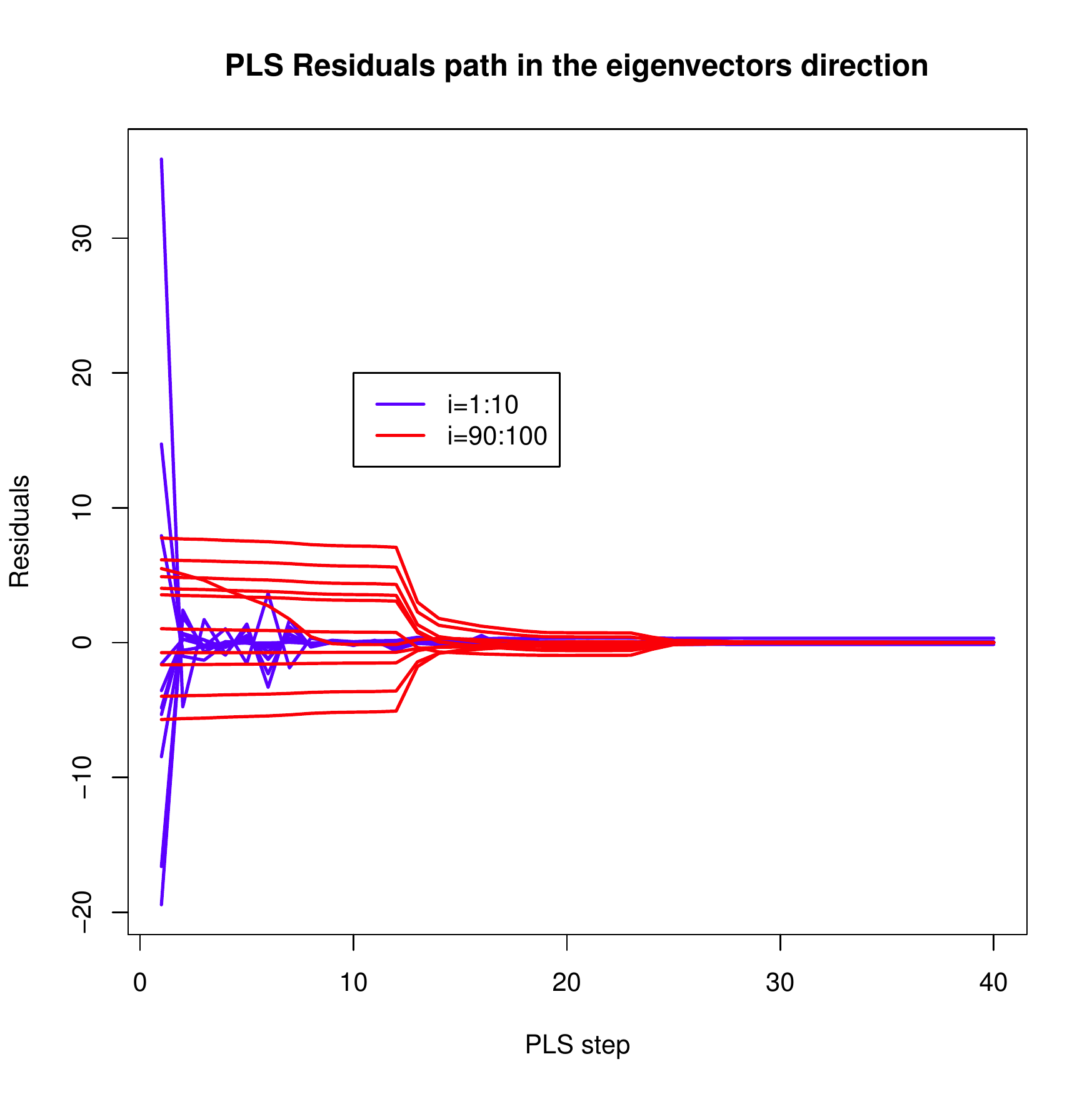} 
\caption{}
  \end{minipage}
\end{figure}

\begin{prop}
\label{prop:residual-closed}
Let $n$ and $k$ fixed and $i\in \llbracket 1,n \rrbracket $. 
If $\lambda_{j}= \lambda_{i}+\delta$ then 
$$ \mid \hat{Q}_{k}(\lambda_{i})-\hat{Q}_{k}(\lambda_{j})\mid\leq \delta\underset{I_{k}^{+}}{\textrm{max}}\left[ \sum_{l=1}^{k}\frac{1}{\lambda_{j_{l}}}\prod_{m \neq l}\left(1-\frac{\lambda_{i}}{\lambda_{j_{m}}} \right)\right]  +O(\delta^{2}).$$
 \end{prop}
 \begin{proof}
 Let $n$ and $k$ fixed and $i\in \llbracket 1,n \rrbracket $. Assume that $\lambda_{j}= \lambda_{i}+\delta$. We have 
$$\hat{Q}_{k}(\lambda_{j})=\sum_{(j_{1},..,j_{k})\in I^{+}_{k}}\left[ \hat{w}_{(j_{1},..,j_{k})}\prod_{l=1}^{k}\left( 1-\frac{\lambda_{j}}{\lambda_{j_{l}}}\right) \right]=\sum_{(j_{1},..,j_{k})\in I^{+}_{k}}\left[ \hat{w}_{(j_{1},..,j_{k})}\prod_{l=1}^{k}\left( 1-\frac{\lambda_{i}+\delta}{\lambda_{j_{l}}}\right) \right] .$$
By expanding $\prod_{l=1}^{k}\left( 1-\frac{\lambda_{i}+\delta}{\lambda_{j_{l}}}\right)$ we get 
$$ \hat{Q}_{k}(\lambda_{i})=\hat{Q}_{k}(\lambda_{j})-\delta \sum_{(j_{1},..,j_{k})\in I^{+}_{k}} \hat{w}_{(j_{1},..,j_{k})}\left[ \sum_{l=1}^{k}\frac{1}{\lambda_{j_{k}}}\prod_{m \neq l}\left(1-\frac{\lambda_{i}}{\lambda_{j_{m}}} \right)\right] +O(\delta^{2}).$$
Then using the fact that $\sum_{(j_{1},..,j_{k})\in I^{+}_{k}} \hat{w}_{(j_{1},..,j_{k})}=1$ we deduce Proposition \ref{prop:residual-closed}.
 \end{proof}
Thus for nearby eigenvalues the filter factors are almost the same
Therefore if $\hat{Q}_{k}(\lambda_{i})$ is small then $\hat{Q}_{k}(\lambda_{j})$ will be small too if $\lambda_{j}$ is closed enough to $\lambda_{i}$.
In particular if the eigenvalues are clustered into $k$ groups and if the residuals associated to the center of the clusters are close to zero then all the residuals will be closed to zero too.

The expression of the residuals provided by Theorem \ref{theo: expression-det-2} will be very useful and central elsewhere in this paper to further explore the PLS method and prove new statistical results.

\subsection{Filter factors and shrinkage properties}
\label{subsection:shrink}
In this subsection we show that we recover some of the results first proved by \cite{BUTLER00} and \cite{LIN00} on the shrinkage properties of the PLS estimator and more particularly on its expansion or contraction in the eigenvectors directions using the expression of the residuals provided by Theorem \ref{theo: expression-det-2}.
We have (see \cite{LIN00})
$$\hat{\beta}_{k}=\sum_{i=1}^{n}f_{i}^{k}\frac{\hat{p}_{i}}{\sqrt{\lambda}_{i}}v_{i}$$
where the elements $f_i^{k}=1-\hat{Q}_{k}(\lambda_i)$ are called the filter factors. 
In their study Lingjaerde and Christophersen use the following implicit expression of $\hat{Q}_{k}$ 
$$\hat{Q}_{k}(t)=\dfrac{(\theta_{1}^{(k)}-t)...(\theta_{k}^{(k)}-t)}{\theta_{1}^{(k)}...\theta_{k}^{(k)}}$$
where $(\theta_{i}^{(k)})_{1\leq i\leq n}$ are the eigenvalues of $W_{k}({W_{k}}^{T}\Sigma W_{k}){W_{k}}^{T}$ (also  called the Ritz eigenvalues) to study the shrinkage properties of PLS. \cite{LIN00} showed that all the PLS shrinkage factors are not in $\left[0,1 \right]$ and can be larger than one. They even proved more precisely that the filter factors oscillate between below and above one (depending on the parity of the index of the factors). We recover these results from our expression of the residuals  provided in Theorem \ref{theo: expression-det-2}
$$\hat{Q}_{k}(\lambda_{i})=
\sum_{(j_{1},..,j_{k})\in I^{+}_{k}}\left[\hat{w}_{(j_{1},..,j_{k})} \prod_{l=1}^{k}(1-\frac{\lambda_{i}}{\lambda_{j_{l}}})\right],$$
where we recall that $\hat{w}_{(j_{1},..,j_{k})}:= \dfrac{\hat{p}_{j_{1}}^{2}...\hat{p}_{j_{k}}^{2}\lambda_{j_{1}}^{2}...\lambda_{j_{k}}^{2}V(\lambda_{j_{1}},...,\lambda_{j_{k}})^{2}}{\sum_{(j_{1},..,j_{k})\in I^{+}_{k}} \hat{p}_{j_{1}}^{2}...\hat{p}_{j_{k}}^{2}\lambda_{j_{1}}^{2}...\lambda_{j_{k}}^{2}V(\lambda_{j_{1}},...,\lambda_{j_{k}})^{2}}$.
From this formula we deduce $$f_{i}^{k}=\sum_{(j_{1},..,j_{k})\in I^{+}_{k}} \hat{w}_{(j_{1},..,j_{k})}\left[ 1-\prod_{l=1}^{k}(1-\frac{\lambda_{i}}{\lambda_{j_{l}}}) \right].$$ 
Notice that the filter factors are completely and explicitely determined by the spectrum and the eigenvectors of $X^{T}X$.

If $k<n$ and $i=n$ then $0<\prod_{l=1}^{k}(1-\frac{\lambda_{n}}{\lambda_{j_{l}}})<1$ and from $$\sum_{(j_{1},..,j_{k})\in I^{+}_{k}}\hat{w}_{(j_{1},..,j_{k})}=1$$ we conclude that $0<f_n^{k}<1$.

If $k<n$ and $i=1$ then 	
\begin{equation}
\lbrace 
\begin{array}{ccc}
\prod_{l=1}^{k}(1-\frac{\lambda_{1}}{\lambda_{j_{l}}})<0 & \mbox{if} & k \: \mbox{is odd} \\
\prod_{l=1}^{k}(1-\frac{\lambda_{1}}{\lambda_{j_{l}}})>0 & \mbox{if} & k \: \mbox{is even}
\end{array}
\end{equation}
and
\begin{equation}
\lbrace
\begin{array}{ccc}
f_1^{k}>1 & \mbox{if} & k \: \mbox{is odd} \\
f_1^{k}<1 & \mbox{if} & k \: \mbox{is even}.
\end{array}
\end{equation}
For the other filter factors we can have $f_i^{k}\leq 1$ or $f_i^{k}\geq 1$ (depending on the distribution of the spectrum) contrary to the PCR  or Ridge filter factors which always lies in $\left[0,1 \right] $. Therefore PLS shrinks in some directions and expands in others. However the PLS estimator is considered as a shrinkage estimator because $ \Vert \hat{\beta}_{k}^{PLS}\Vert\leqslant\Vert \hat{\beta}_{OLS}\Vert $ (see \cite{GOU96}).

We also recover Theorem 7 of \cite{LIN00}. Indeed if we have $\lambda_{i}< \lambda_{n}(1+\sqrt{\epsilon})$ then a straightforward calculation using formula (\ref{eq:final-expression-residuals}) leads to 
$f_i^{k}< 1+\epsilon$.

\section{Bounds for the empirical risk and prediction error}
\label{section:error}
In this section we further explore the statistical properties of PLS. For this, we investigate the accuracy of PLS through the study of the empirical risk and the least square error of prediction which are two criteria commonly used for assessing the quality of an estimator. 

From now, on we assume that the $(\varepsilon_{i})_{1\leq i \leq n}$ are i.i.d centered random variables with commmon variance $\sigma^{2}$ and for simplicity we also assume that the observations on the $X$ variables are centered and normalized, that is $\frac{1}{n}\sum_{i=1}^{n}X_{ij}=0$ and $\frac{1}{n}\sum_{i=1}^{n}X_{ij}^{2}=1$.

\subsection{Empirical risk}
The following proposition provides an upper bound for the MSE (mean square error). The MSE quantifies the fit of the model to the data set used. 
\begin{prop}
\label{prop:empirical-risk}
We have for $k<n$
\begin{equation}
\mathbb{E}\left[  \frac{1}{n}\Vert Y-X\hat{\beta}_{k}\Vert^{2}\right] \leqslant \left[  \frac{1}{n}\left(\frac{\sqrt{C(X^{T}X)}-1}{\sqrt{C(X^{T}X)}+1} \right) ^{2k}\parallel X\beta^{*}\parallel^{2}\right] \left[ 1 + \dfrac{n\sigma^{2}}{\parallel X\beta^{*}\parallel^{2}}\right]
\label{eq:empirical risk}
\end{equation}
where $C(X^{T}X)=\frac{\lambda_{1}}{\lambda_{n}}$ is the ratio of the two extreme non zero eigenvalues of $X^{T}X$. 

Obviously, if $k=n$, $$\mathbb{E}\left[  \frac{1}{n}\Vert Y-X\hat{\beta}_{n}\Vert^{2}\right]=0.$$
\end{prop}

The first factor in equation (\ref{eq:empirical risk}) represents the error due only to the regularization if no noise (projection of $X\beta^{*}$ onto the Krylov subspace $\mathcal{K}^{k}(X^{T}X,X^{T}X\beta^{*})$  and in the second factor $\dfrac{n\sigma^{2}}{\parallel X\beta^{*}\parallel^{2}}$ represents the inverse of the signal to noise ratio. 
We can notice that the upper bound relies on $\parallel X\beta^{*}\parallel^{2} = \sum_{i=1}^n \lambda_i \tilde{\beta}_i^2$. This term links the regularity of $\beta^{*}$ with the decay of the eigenvalues of $X^TX$. It thus can be seen as a Source Condition, see for instance in \cite{MR1408680}. We can state a result similar to the one of Proposition \ref{prop:empirical-risk} replacing $\Vert. \Vert_{2}$ by $\Vert .\Vert_{p}$, $p\in \mathbb{N}^{*}$.

We can notice that the convergence of the empirical risk  is associated with upper bounds derived using scaled and shifted Chebyshev polynomials. In fact the key of the proof of Proposition \ref{prop:empirical-risk} is essentialy based on the following proposition
\begin{prop}\cite{SAAD92}\\
\label{theo:chebychev}
Let $\left[\alpha,\beta \right] $ be a non empty interval in $\mathbb{R}$ and let $\gamma$ be any scalar such with $\gamma \notin \left]\alpha,\beta \right[ $. We define
 $\mathcal{E}_{k}:=\left\lbrace  \textrm{P polynomial of degree} \:k  \:\textrm{with}\: P(\gamma)=1 \right\rbrace $.
 
Then the minimum $\underset{P\in \mathcal{E}_{k}}{\rm{min}}\underset{t \in \left[\alpha,\beta \right] }{\rm{max}}\vert P(t)\vert$ is reached by the polynomial 
$$\hat{C}_{k}(t)= \dfrac{C_{k}(1+2\frac{t-\beta}{\beta-\alpha})}{C_{k}(1+2\frac{\gamma-\beta}{\beta-\alpha})},$$
where $C_{k}$ is the $k^{th}$ Chebychev polynomial i.e., for $x\in \left[-1,1 \right]  $, $$C_{k}(x)=\frac{1}{2}\left((x-\sqrt{x^{2}-1})^{k}+(x+\sqrt{x^{2}-1})^{k} \right) .$$
 
The maximum of $C_{k}$ for $x \in \left[ -1,1\right] $ is 1 and 
$$\underset{ P\in \mathcal{E}_{k} }{\rm{min}}\underset{t \in \left[\alpha,\beta \right] }{\rm{max}}\vert P(t)\vert=\dfrac{1}{\vert C_{k}(1+2\frac{\gamma-\beta}{\beta-\alpha})\vert}=\dfrac{1}{\vert C_{k}(2\frac{\gamma-\mu}{\beta-\alpha})\vert}$$
with $\mu=\frac{\alpha+\beta}{2}$.
\end{prop}
 
 Notice that the convergence rate of the empirical risk is exponential in $k$  with respect to the ratio between the maximum and the minimum of the non zero eigenvalues. In fact $\left \lvert\frac{\sqrt{C(X^{T}X)}-1}{\sqrt{C(X^{T}X)}+1} \right \rvert< 1$ and is equal to zero if and only if all the eigenvalues are the same.  Therefore the closer to one is the condition number $C(X^{T}X)$ the faster is the decrease of the empirical risk with respect to $k$, in an exponential way while it turns polynomial for Ridge Regression for instance.
 
Figure \ref{figure12} represents the empirical risk for different values of the level of noise.
  \begin{figure}[H]
  \centering
   \includegraphics[scale=0.5]{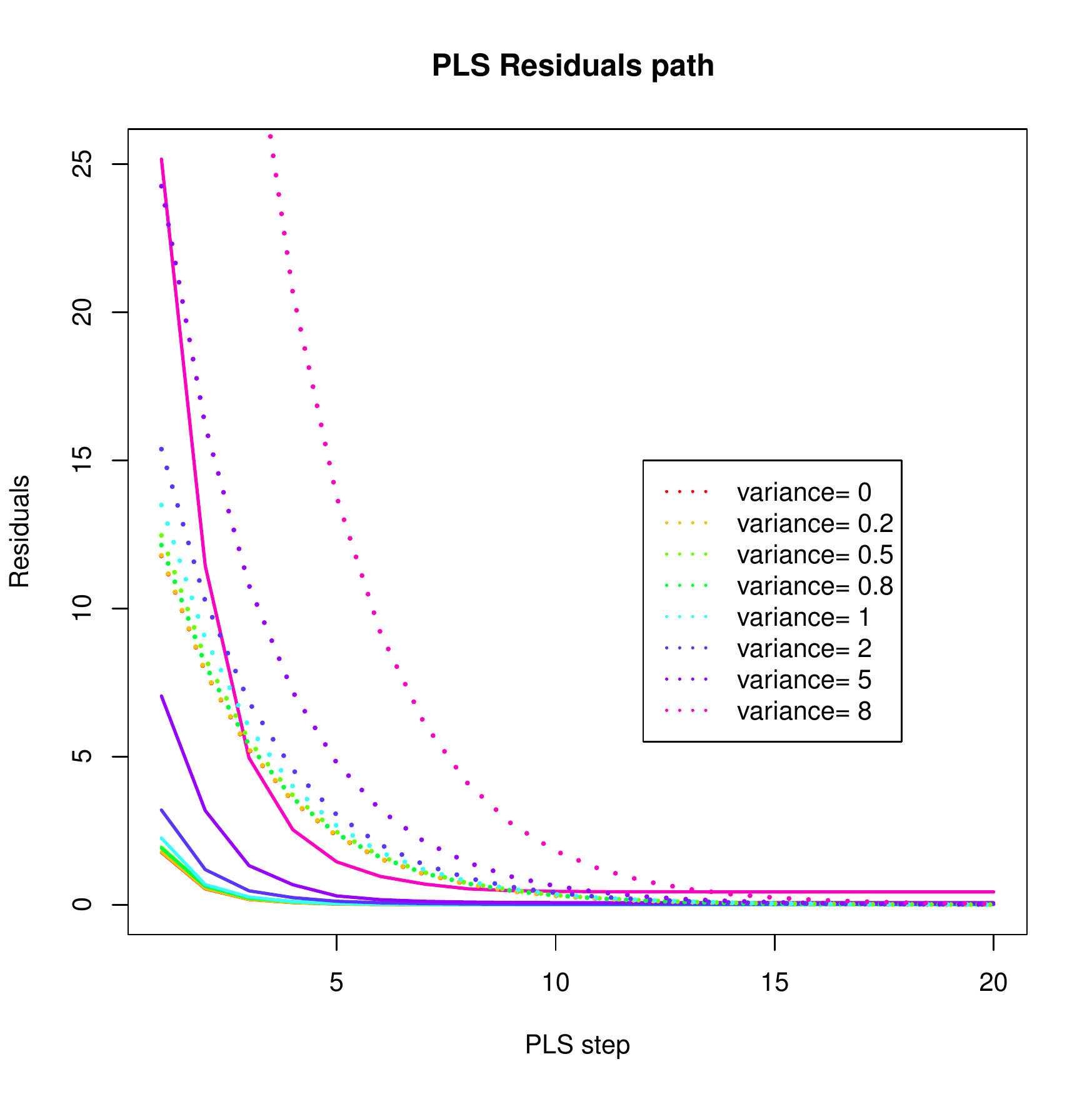} 
   \caption{}
   \label{figure12}
\end{figure}

A way to improve the method could be the use of preconditioners i.e. of a matrix $M$ used to convert the problem $X^{T}Y=X^{T}X\beta^{*}+X^{T}\varepsilon$ into another equivalent problem i.e. into $M^{-1}X^{T}X\beta=M^{-1}X^{T}(Y-\varepsilon)$ in such a way that it increases the rate of convergence.

As noticed below if there are only $k$ distincts eigenvalues or if the contribution of $Y$ is only non zero along $k$ eigenvectors then we also have 
$\mathbb{E}\left[  \frac{1}{n}\Vert Y-X\hat{\beta}_{k}\Vert^{2}\right]=0.$
This is a straigthforward consequence of equation  (\ref{eq:residu-poly}) in Proposition \ref{prop:poly}, taking $Q(x)=\prod_{i=1}^{k}\left( 1-\frac{x}{{\overline{\lambda}_{i}}}\right) $ where $({\overline{\lambda}_{i}})_{1\leq i\leq k}$ are the representatives of respectively the different non zero eigenvalues and the eigenvalues associated to a non zero contribution of the response to the associated eigenvectors.
In the same way the empirical risk will be very small at step $k$ if the eigenvalues are clustered into $k$ groups. 

To illustrate this particular behaviour of the PLS estimator we have performed some simulations.
The data sets are simulated according to model (\ref{eq:regression-model}) with $n=p=100$. The best latent components are choosen using the function \texttt{pls.regression.cv}.
For the first simulation below we consider that the eigenvalues are partitionned into $2$ clusters and for the second one that the eigenvalues are partitionned into $10$ clusters.

 \begin{figure}[H]
 \begin{minipage}[b]{0.5\linewidth}
   \centering
   \label{Figure5}
       \includegraphics[scale=0.4]{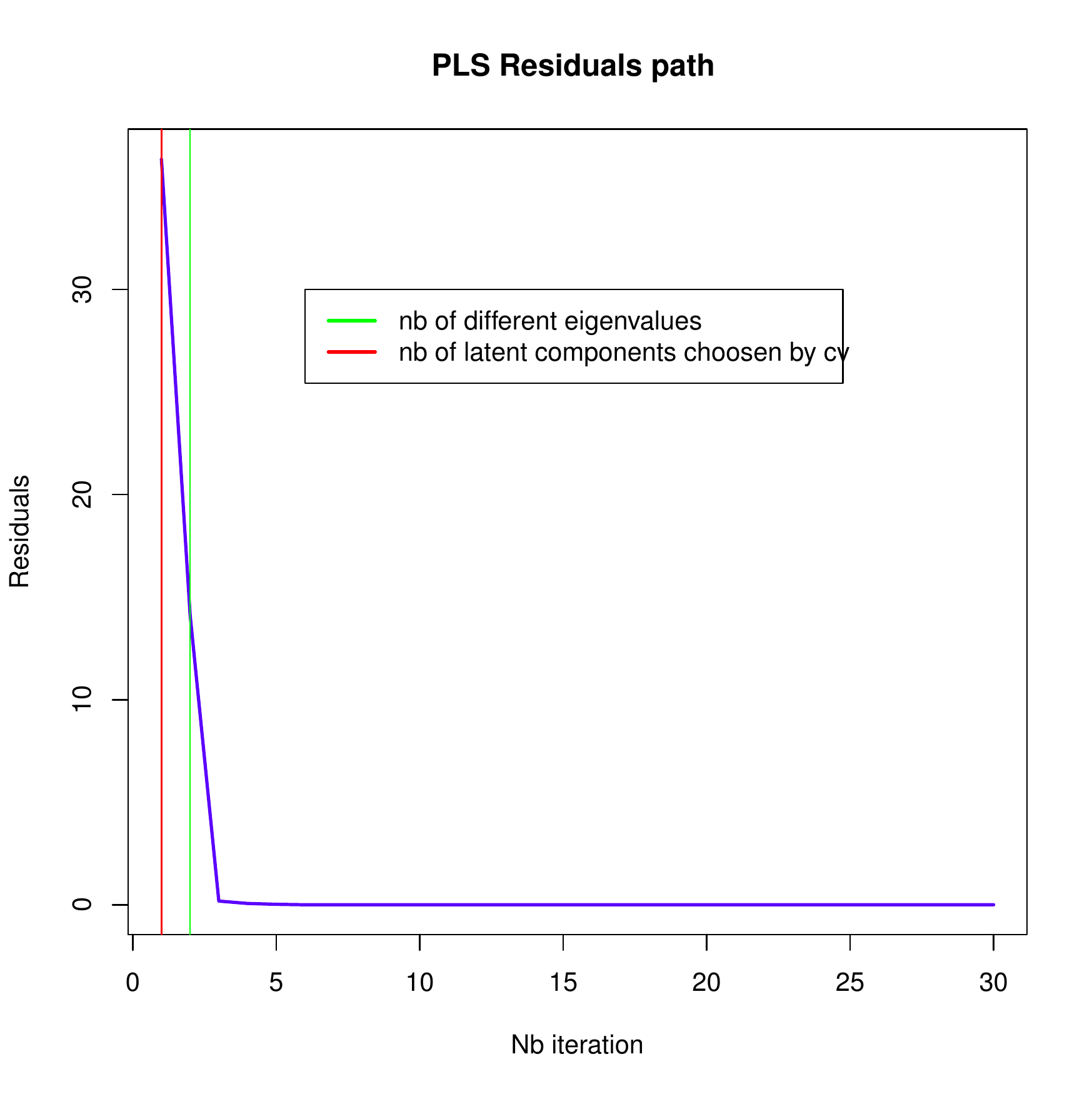} 
\caption{}
  \end{minipage}
\begin{minipage}[b]{0.5\linewidth}
   \centering
   \includegraphics[scale=0.4]{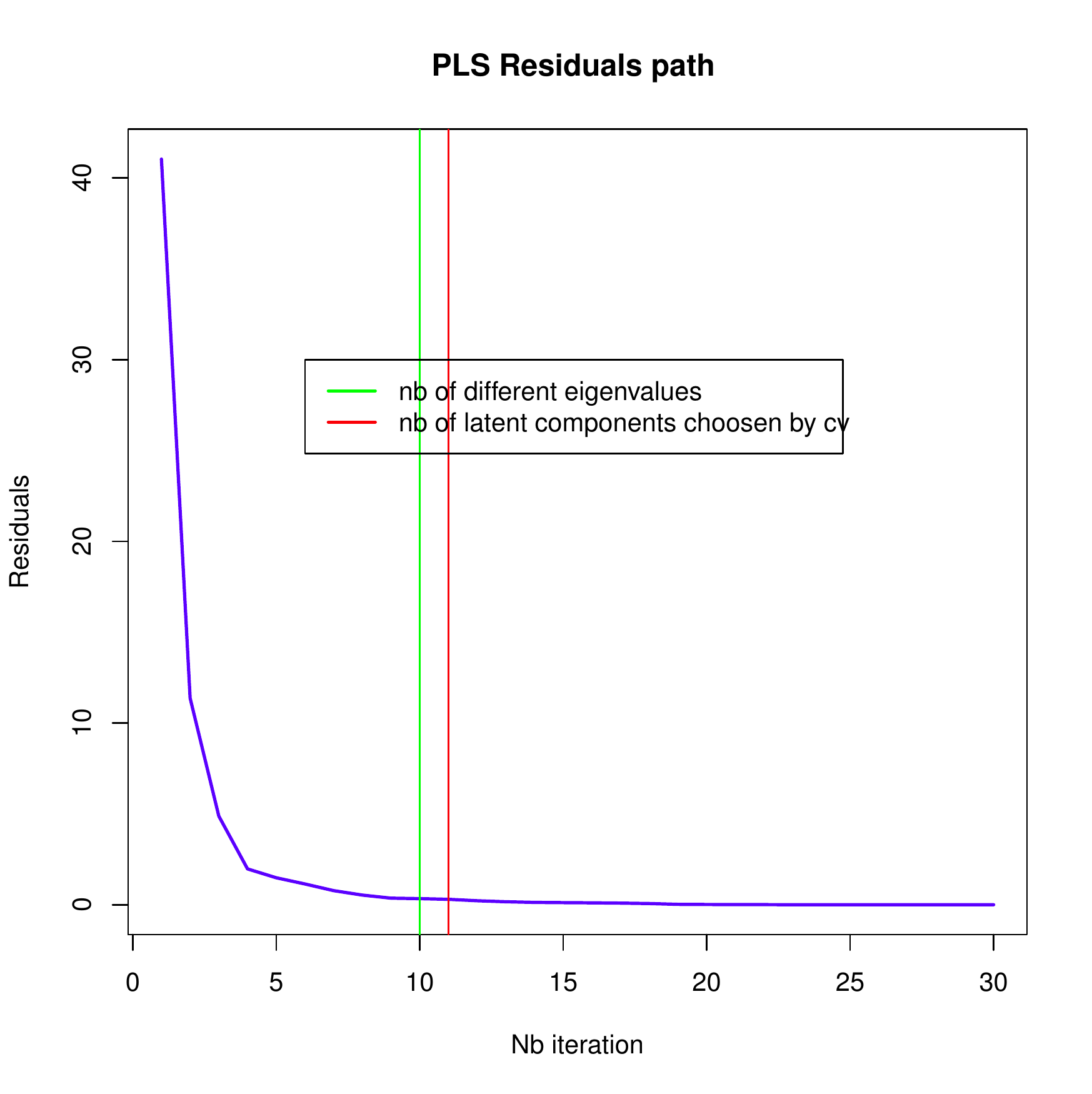} 
\caption{}
\label{Figure4}
  \end{minipage}
\end{figure}

 For the next three models we consider that $c$ eigenvalues are between 2 and 20 and all the others very close to zero (between 0.1 and 0.5) with $c$ respectively equals to $5,10$ and $15$.
The residuals $\parallel Y-X\hat{\beta}_{k}\parallel^{2}$ are plotted for different values of $k$.
\begin{figure}[H]

  \begin{minipage}[b]{0.5\linewidth}
   \centering
       \includegraphics[scale=0.4]{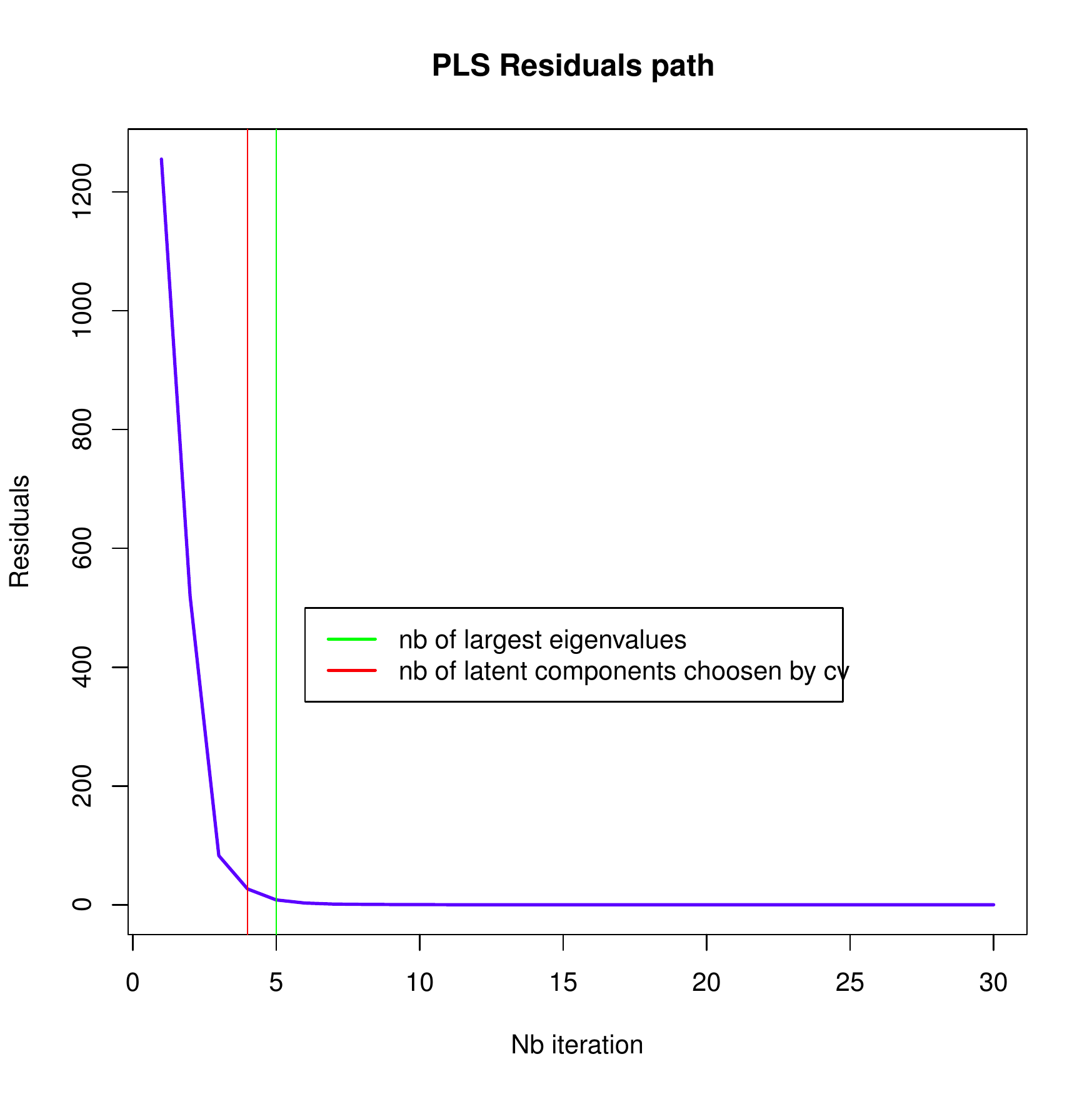} 
\caption{}
  \end{minipage}
\begin{minipage}[b]{0.5\linewidth}
   \centering
   \includegraphics[scale=0.4]{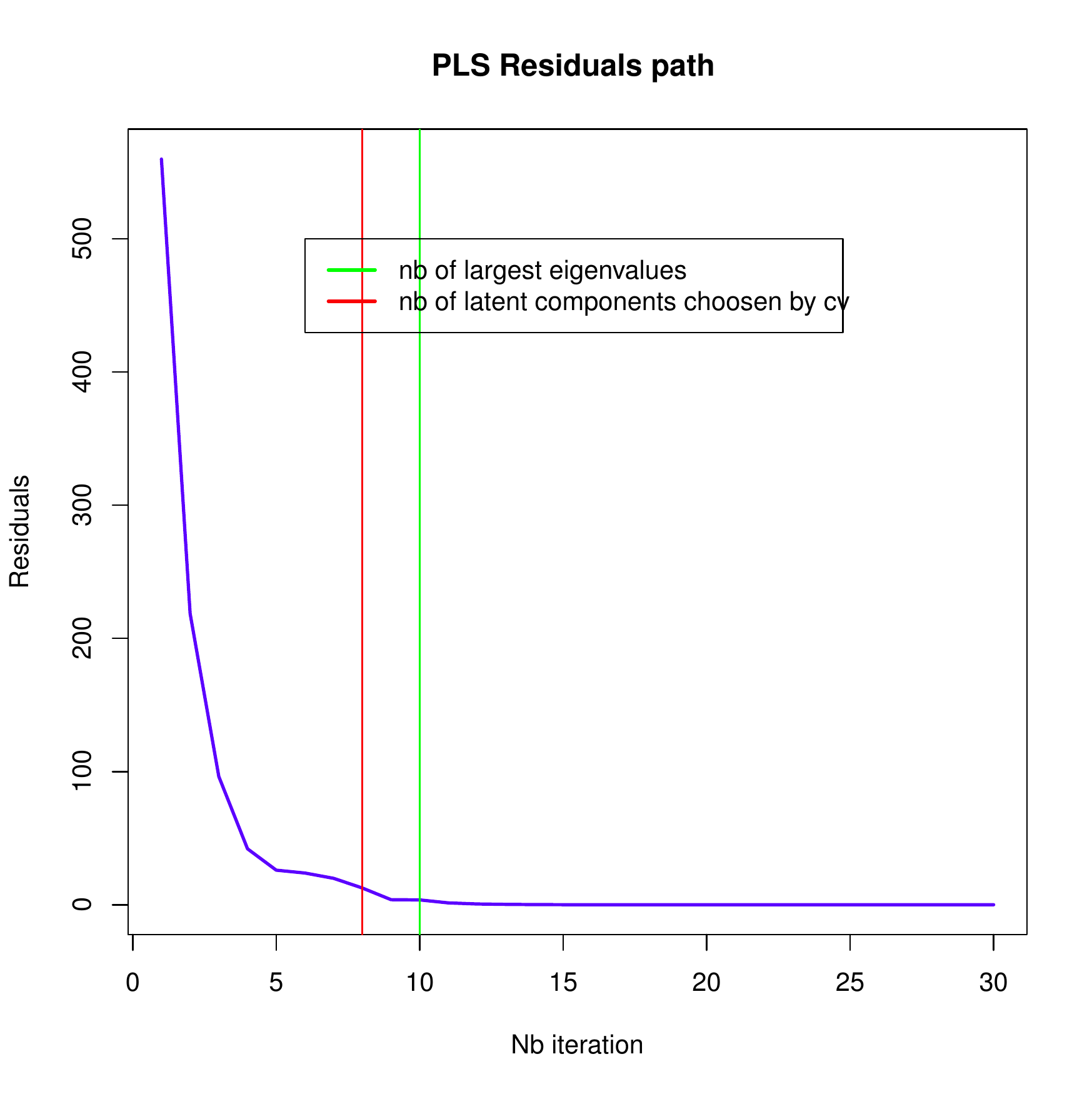} 
\caption{}
  \end{minipage}

\end{figure}

\begin{figure}[H]

\begin{minipage}[b]{0.10\linewidth}
 \centering
   \includegraphics[scale=0.4]{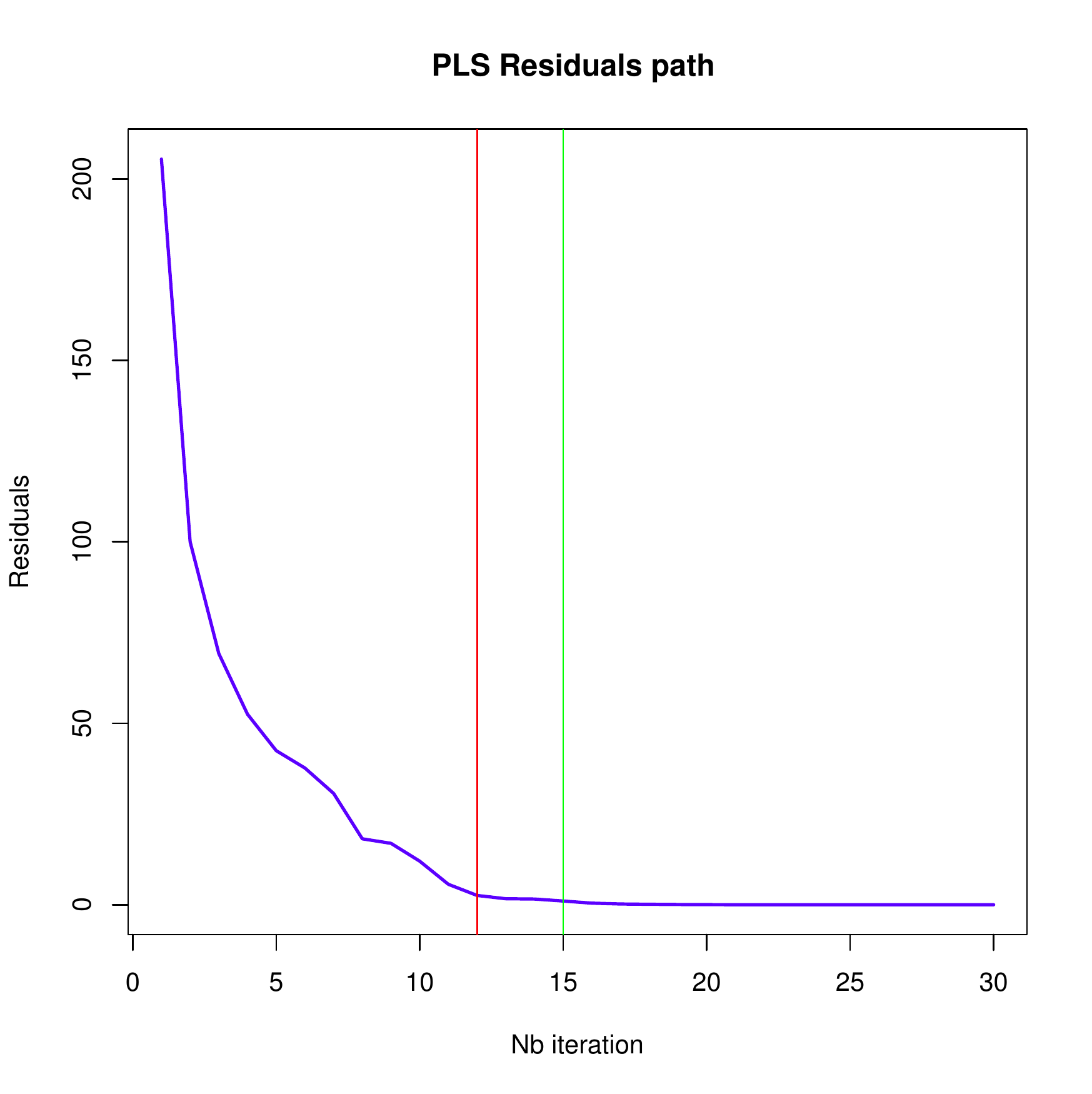} 
\caption{}
  \end{minipage}
  \end{figure}
Figures above show that when there is a clear gap in the distribution of the eigenvalues with $k$ large eigenvalues and the others very small then there is no need to go very far to recover most of the information with the PLS estimator. It is mainly due to the fact that the small eigenvalues can be considered as belonging to a same cluster.
 
\subsection{Prediction error}
Here we study the distance between the estimator and the true parameter in term of prediction error $\frac{1}{n}\parallel X\beta^{*}-X\hat{\beta}_{k} \parallel^{2} $. The expression of the prediction error is not as simple as the one of PCR and thus an upper bound for the prediction error is not as obvious since the PLS procedure is not a linear modeling procedure. Indeed, the direction of the new subspace onto which we project the observations depends in a complicated way on the singular value spectrum of the design matrix and also on the response. To compute or bound the prediction error we have to be careful because this also implies a control of the error due to the randomness of the subspace onto which we project your observations. We are going to use formula (\ref{eq:final-expression-residuals}) of Theorem \ref{theo: expression-det-2} to study the prediction error. 

We first make the following assumptions.
The real variables $\varepsilon_{1},...,\varepsilon_{n}$ are assumed to be unobservable i.i.d centered gaussian random variables with common variance $\sigma^{2}_{n}$.
We also assume
\begin{itemize}
\item (H.1): $\sigma^{2}_{n}=\mathcal{O}(\frac{1}{n})$. In other words we assume that 
$Y_{i}=x_{i}^{T}\beta^{*}+\delta_{n}\varepsilon_{i}$
where $\varepsilon_{i} \sim \mathcal{N}(0,1)$ and $\delta_{n} =\frac{1}{\sqrt{n}}$ is the noise level which is related to the number of observations.
\item (H.2): there exists a constant $L>0$ such that $\underset{1\leq i\leq n}{\textrm{min}}\lbrace p_{i}^{2}\rbrace\geq L,$ where $p_{i}=(X\beta^{*})^{T}u_{i}$. 
\end{itemize}
These two assumptions warrant that the signal to noise ratio $\left\lbrace  \left  \lvert \frac{\tilde{\tilde{\varepsilon_{i}}}}{p_{i}}\right \rvert \right\rbrace _{1\leq i \leq n} $ is not too small. This last quantity will appear again many time thereafter.

To bound from above the prediction error we have to be careful because PLS is a projected method but not onto a fixed subspace. The Krylov subspace onto which we project the data depends on $Y$ and thus is a random subspace. Therefore we have to also control the randomness of the subspace onto which we project the data. To do so, we introduce an oracle which is the regularization $\beta_{k}$ of $\beta^{*}$ onto the noise free Krylov subspace of dimension $k$.
This regularized approximation of $\beta^{*}$ is defined as
$$\beta_{k}\in \underset{\beta \in \mathcal{K}^{k}}{\rm{argmin}} \Vert X\beta^{*}-X\beta\Vert^{2}$$
where $\mathcal{K}^{k}:=\mathcal{K}^{k}(X^{T}X,X^{T}X\beta^{*}) $ is the noise free Krylov subspace.
Therefore we have
$$
\beta_{k}=P^{*}_{k}(X^{T}X)X^{T}X\beta^{*}
$$
where $P^{*}_{k}$ is a polynomial of degree $k-1$ which satisfies $$\Vert X\beta^{*}-XP^{*}_{k}(X^{T}X)X^{T}X\beta^{*}\Vert^{2}=\underset{P \in \mathcal{P}_{k-1}}{\textrm{argmin}}\Vert X\beta^{*}-XP(X^{T}X)X^{T}Y\Vert^{2}$$
and $X\beta^{*}-X\beta_{k}=Q^{*}_{k}(XX^{T})X\beta^{*}$
with $Q^{*}_{k}(t)=1-tP^{*}_{k}(t) \in \mathcal{P}_{k,1}$.
Then by the same arguments as the ones used to prove Proposition \ref{prop: poly-ortho} and Theorem \ref{theo: expression-det-2} we have that
\begin{enumerate}
\item the sequence of polynomials $(Q^{*}_{k})_{1\leq k \leq n }$ are orthogonals with respect to the measure $$d\mu(\lambda)=\sum_{j=1}^{n}\lambda_{j}p_{j}^{2}\delta_{\lambda_{j}},$$
where $p_{j}:=u_{j}^{T}(X\beta^{*})$.
\item $$
Q^{*}_{k}(\lambda_{i}):=\sum_{(j_{1},..,j_{k})\in I^{+}_{k}}w_{(j_{1},...,j_{k})}\prod_{l=1}^{k}(1-\frac{\lambda_{i}}{\lambda_{j_{l}}}) $$
where $w_{(j_{1},...,j_{k})}:=  \dfrac{p_{j_{1}}^{2}...p_{j_{k}}^{2}\lambda_{j_{1}}^{2}...\lambda_{j_{k}}^{2}V(\lambda_{j_{1}},...,\lambda_{j_{k}})^{2}}{\sum_{(j_{1},..,j_{k})\in I^{+}_{k}} p_{j_{1}}^{2}...p_{j_{k}}^{2}\lambda_{j_{1}}^{2}...\lambda_{j_{k}}^{2}V(\lambda_{j_{1}},...,\lambda_{j_{k}})^{2}}$.
\end{enumerate}

In fact we can write $$\frac{1}{n}\parallel X\beta^{*}-X\hat{\beta}_{k}\parallel^{2} = \frac{1}{n}\parallel X\beta^{*}-X\hat{P}_{k}(X^{T}X)X^{T}Y\parallel^{2}$$
$$\leq \frac{2}{n}\parallel X\beta^{*}-XP^{*}_{k}(X^{T}X)X^{T}Y\parallel^{2}+\frac{2}{n}\parallel XP^{*}_{k}(X^{T}X)X^{T}Y-X\hat{P}_{k}(X^{T}X)X^{T}Y\parallel^{2}$$
\begin{equation}
\label{eq:prediction_error}
=\frac{2}{n}\parallel X\beta^{*}-XP^{*}_{k}(X^{T}X)X^{T}Y\parallel^{2}+\frac{2}{n}\parallel \left( \hat{Q}_{k}(XX^{T})-Q^{*}_{k}(XX^{T})\right) Y\parallel^{2}.
\end{equation}
Therefore to bound by above $\frac{1}{n}\parallel X\beta^{*}-X\hat{\beta}_{k}\parallel^{2} $ we need to control two other quantities. The first one represents the error of regularization when projecting the linear predictor plus the noise on the observations onto the noise free Krylov subspace. The second quantities represents the approximation error between the projection onto the noise free Krylov subspace and onto the random Krylov subspace built from the observations.

Now let us introduce our main result on prediction error which provides an upper bound for the prediction error assuming a low variance of the observations noise.
\begin{theorem}
\label{theo:prediction_error}
Let $k<n$ and assume that (H.1) and (H.2) holds.
Then, with probability at least $1-n^{1-C}$ where $C>1$, we have 

$$
\frac{1}{n}\parallel X\beta^{*}-X\hat{\beta}_{k}\parallel^{2}\leq $$ 
$$\frac{1}{n}\left[ 2\left(\frac{\sqrt{C(X^{T}X)}-1}{\sqrt{C(X^{T}X)}+1} \right) ^{2k}+4\dfrac{\log(n)}{nL}\left(1+ \left(  \frac{\sqrt{C(X^{T}X)}-1}{\sqrt{C(X^{T}X)}+1}\right) ^{2k} \right) \right] \Vert X\beta^{*}\Vert^{2}$$
$$+\frac{4k^{2}\tilde{C}^{2}}{L}\frac{\log n}{n^{2}}\left(1+C\sqrt{\frac{\log n}{nL}} \right)^{2} \Vert X\beta^{*}\Vert^{2}_{W},$$
where $C(X^{T}X)=\frac{\lambda_{1}}{\lambda_{n}}$, $\tilde{C}$ is a constant and $W=\underset{1\leq i \leq n}{\rm{diag}}\left( \underset{I_{k}^{+}}{\rm{max}}\left( \prod_{l=1}^{k}\left \lvert \frac{\lambda_{i}}{\lambda_{j_{l}}}-1\right \rvert^{2}\right) \right)$.
\end{theorem}

\begin{proof}

Theorem \ref{theo:prediction_error} is a straightforward consequence of Proposition \ref{prop: first term}, Proposition \ref{prop:pred-second-term} and (\ref{eq:prediction_error}) below. 
In fact we recall that $$\frac{1}{n}\parallel X\beta^{*}-X\hat{\beta}_{k}\parallel^{2} $$
$$\leq\frac{2}{n}\parallel X\beta^{*}-XP^{*}_{k}(X^{T}X)X^{T}Y\parallel^{2}+\frac{2}{n}\parallel \left( \hat{Q}_{k}(XX^{T})-Q^{*}_{k}(XX^{T})\right) Y\parallel^{2}.$$
The following proposition provides an upper bound for the first term of (\ref{eq:prediction_error}).
\begin{prop}
\label{prop: first term}
With probability at least $1-n^{1-C}$ where $C>1$, we have for all $i=1,...,n$
$$\frac{1}{n}\parallel X\beta^{*}-XP^{*}_{k}(X^{T}X)X^{T}Y\parallel^{2} $$
$$
\leq \frac{1}{n}\left[ 2\left(\frac{\sqrt{C(X^{T}X)}-1}{\sqrt{C(X^{T}X)}+1} \right) ^{2k}+4\dfrac{\log(n)}{nL}\left(1+ \left(  \frac{\sqrt{C(X^{T}X)}-1}{\sqrt{C(X^{T}X)}+1}\right) ^{2k} \right)  \right]  \Vert X\beta^{*}\Vert^{2}.$$
\end{prop}

Then we bound by above the second term $\frac{1}{n}\parallel \left( \hat{Q}_{k}(XX^{T})-Q^{*}_{k}(XX^{T})\right) Y\parallel^{2}$.
\begin{prop}
\label{prop:pred-second-term}
Assume (H.1) and (H.2).
Then with probability larger than $1-n^{1-C}$ where $C>1$ we have
$$\frac{1}{n}\parallel\left( \hat{Q}_{k}(X^{T}X)-Q^{*}_{k}(X^{T}X)\right) Y\parallel^{2}$$
$$\leq \frac{4k^{2}\tilde{C}^{2}}{L}\frac{\log n}{n^{2}}\left(1+C\sqrt{\frac{\log n}{nL}} \right)^{2} \Vert X\beta^{*}\Vert^{2}_{W},$$
where $W=\underset{1\leq i \leq n}{\rm{diag}}\left( \underset{I_{k}^{+}}{\rm{max}}\left( \prod_{l=1}^{k}\left \lvert \frac{\lambda_{i}}{\lambda_{j_{l}}}-1\right \rvert^{2}\right) \right)$ and $\tilde{C}$ is a constant.
\end{prop}
The theorem is proved by combining the two previous bounds.

\end{proof}

The bound in Theorem \ref{theo:prediction_error} highly depends on the signal to noise ratio which must not be too small with respect to the eigenvector directions of $X^{T}X$ to ensure good statistical properties of the PLS estimator. This is the major difference with PCA that takes into account the variance of the noise on the observations to build the latent variables but not the level of the signal. On the contrary PLS takes into account the signal through $Y$ to construct the latent variables. That is why for PLS the signal to noise ratio plays an important role in the accuracy of the  model. 

The following simulation highlights this statement showing that there is generally no hope to recover a good approximation of the predicted function in case of a high variance of the noise. 
We have compared the performances of the PLS estimator at different steps and for different levels of noise. Here is a typical example of the behaviour of the prediction error for the PLS estimator. The data set is simulated according to model (\ref{eq:regression-model}) with $n=p=100$. Figure \ref{Figure1} represents the PLS prediction error path for different value of the paremeter $k$ and for different level of noise on the observations.

\begin{figure}[H]
   \centering
       \includegraphics[scale=0.5]{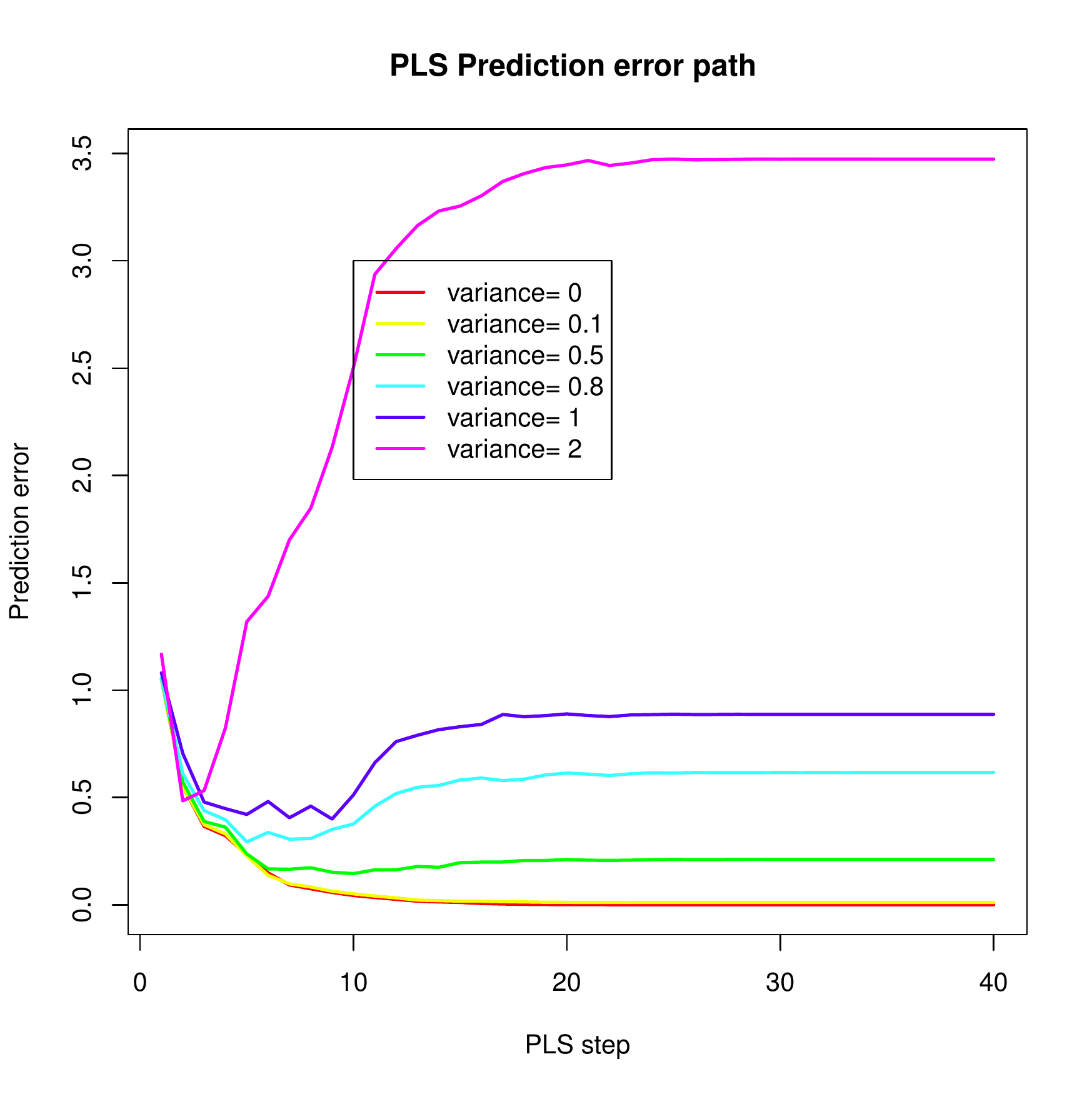} 
       \caption{}
       \label{Figure1}
\end{figure}

Figures above show that there is no assurance that the prediction error goes to zero when the variance is too high compared to the number of observations. This is essentially due to the fact that the PLS method is an iterative technique and thus the noise can propagate at each step of the construction leading to an undue amplification of the error.

\section{Conclusion}
PLS is a method used to remove multicollinearities and based on the construction of a new subspace of reduced dimension. This new subspace is built to maximize both the covariance of the covariates and the correlation to the response. They key idea behind PLS is to approximate and regularize the pseudo-inverse of the covariance matrix by a polynomial in the power of the matrix. PLS is in fact a least square problem with random linear constraints. This method can also be viewed as a minimization problem over a particular polynomial subspace. From this perspective we showed that the PLS residuals are in fact orthogonal polynomials with respcect to a measure based on the spectrum of the covariance matrix. From the definition of this discrete measure we deduce a new formula for the residuals. This formula depends explicitely on the observations noise and on the spectrum of the covariance matrix. At last we have taken advantages of these findings in a regression context to state new results for the estimation and prediction error for PLS under a low variance of the noise. The control of the signal-to-noise ratio and of the spectrum distribution seems to be the key to state such results. We have showed that PLS is not an automatic solution to avoid the problem of multicolinearity in regression. We have to be careful when using PLS because its statistical properties are strongly depending on the features of the data and in particular on the distribution of the spectrum. The main drawback of PLS is the fact that it seems inapropriate if $Y$ has too much variation but the advantages is that it takes both $X$ and $Y$ into account in the decomposition of $X$ contrary to PCR. To conclude this paper throw new lights on PLS in a regression context but this is not the end of the road and the formula for the residuals should be explored further to completely understand the method.

\section{Proof}
\subsection{Proof of Proposition \ref{prop: poly-ortho}}
\begin{proof}
Let $k \in \mathbb{N}^{*}$ and $n> l>k$.
Because $\hat{Q}_{k}\in \mathcal{P}_{k,1}$ we have $$XX^{T}\hat{Q}_{k}(XX^{T})Y \in \mathcal{K}^{k+1}(XX^{T},XX^{T}Y).$$ 
Furthermore from (\ref{eq:residu-poly}) we get $\hat{Q}_{l}(XX^{T})Y \perp \mathcal{K}^{l}(XX^{T},XX^{T}Y).$
Besides $\mathcal{K}^{l}(XX^{T},XX^{T}Y)\supset \mathcal{K}^{k+1}(XX^{T},XX^{T}Y).$
Therefore we deduce that for all  $k\neq l$ we have
$X^{T}\hat{Q}_{k}(XX^{T})Y  \perp \hat{Q}_{l}(XX^{T})Y .$
Then using the SVD decomposition of $X$ we get $XX^{T}=\sum_{1\leqslant j\leqslant n}\lambda_{j}u_{i}u_{i}^{T}$ and 
 $$0=\left\langle XX^{T}\hat{Q}_{k}(XX^{T})Y,\hat{Q}_{l}(XX^{T})Y \right\rangle$$
 $$=\left(\sum_{1\leqslant j\leqslant n}\lambda_{j}\hat{Q}_{k}(\lambda_{j})u_{j}u_{j}^{T}Y \right) ^{T}\left(\sum_{1\leqslant j\leqslant n}\hat{Q}_{l}(\lambda_{j})u_{j}u_{j}^{T}Y \right)
  =\sum_{1\leqslant j\leqslant n}\lambda_{j}\hat{Q}_{k}(\lambda_{j})\hat{Q}_{l}(\lambda_{j})(u_{j}^{T}Y)^{2}.
$$

Finally we get
$$0=\sum_{1\leqslant j\leqslant n}\lambda_{j}\hat{Q}_{k}(\lambda_{j})\hat{Q}_{l}(\lambda_{j})(u_{j}^{T}Y)^{2}.$$
And we deduce that $\hat{Q}_{1},\hat{Q}_{2},...,\hat{Q}_{n-1}$ is a sequence of orthonormal polynomials with respect to the measure 
$$d\hat{\mu}(\lambda)=\sum_{j=1}^{n}\lambda_{j}(u_{j}^{T}Y)^{2}\delta_{\lambda_{j}}.$$
\end{proof}

\subsection{Proof of Theorem \ref{theo: expression-det-2}}
We recall that $(\hat{Q}_{1})_{1\leq k <n}$ is a sequence of orthonormal polynomials with respect to the measure 
$d\hat{\mu}(\lambda)$. Returning to the definition of orthogonal polynomials we first express the polynomials $(\hat{Q}_{k})_{1\leqslant k<n}$ as the quotient of two determinants.

\begin{prop}
For all $j \in \mathbb{N}$, let $\hat{m}_{j}=\int x^{j}d\hat{\mu}$ .

Then for all $k \in\llbracket 1,...,n-1 \rrbracket$ we have
\begin{equation}
\hat{Q}_{k}(x)=(-1)^{k}\dfrac{\textrm{det}\left( \hat{G}_{2k-1}(x)\right) }{\textrm{det}\left( \hat{H}_{2k-1}\right) }
\label{prop: expression-poly}
\end{equation}
where $$\hat{G}_{2k-1}(x):= \left[ \begin{array}{cccc}
\hat{m}_{0} & \hat{m}_{1} & ... & \hat{m}_{k} \\ 
 \vdots&  &  &  \\ 
\hat{m}_{k-1} & \hat{m}_{k} & ... & \hat{m}_{2k-1} \\ 
1 & x & ... & x^{k}
\end{array} \right]$$
and   $$\hat{H}_{2k-1}:= \left[ \begin{array}{cccc}
\hat{m}_{1} & \hat{m}_{2} & ... & \hat{m}_{k} \\ 
 \vdots &  &  &  \\ 
\hat{m}_{k-1} & \hat{m}_{k} & & \hat{m}_{2k-2} \\ 
 \hat{m}_{k}&  \hat{m}_{k+1} & ... &  \hat{m}_{2k-1}
\end{array} \right].$$
\end{prop}

\begin{proof}
The polynomials $(\hat{Q}_{k})_{1\leqslant k <n}$ are the ones which satisfy
\begin{enumerate}
\item $\hat{Q}_{k}(x)=\alpha_{k}^{k}x^{k}+\alpha_{k}^{k-1}x^{k-1}+...+\alpha_{1}^{k}x+\alpha_{0}^{k}$
\item $\forall j \in \left[0, k-1 \right] $, $\int \left[  x^{j}(\alpha_{k}^{k}x^{k}+\alpha_{k}^{k-1}x^{k-1}+...+\alpha_{1}^{k}x+\alpha_{0}^{k})\right] d\hat{\mu}=0$
\item $\hat{Q}_{k}(0)=1$
\end{enumerate}
This is equivalent to solve the following system of $k$ equations with $k$ unknowns

$$\forall j \in \llbracket 0, k-1 \rrbracket, \quad \alpha_{k}^{k}\hat{m}_{j+k}+\alpha_{k}^{k-1}\hat{m}_{j+k-1}+...+\alpha_{1}^{k}\hat{m}_{j+1}=-\hat{m}_{j}.$$
The solution $(\alpha_{1}^{k},...,\alpha_{k}^{k})$ of this system satisfies
$$\left[ \begin{array}{cccc}
\hat{m}_{1} & \hat{m}_{2} & ... & \hat{m}_{k} \\ 
 \vdots &  &  &  \\ 
\hat{m}_{k-1} & \hat{m}_{k} &  & \hat{m}_{2k-2} \\ 
 \hat{m}_{k}&  \hat{m}_{k+1} & ... &  \hat{m}_{2k-1}
\end{array} \right]\left[\begin{array}{c}
\alpha_{1}^{k} \\ 
\alpha_{2}^{k} \\ 
\\ 
\alpha_{k}^{k}
\end{array}  \right] =-\left[\begin{array}{c}
\hat{m}_{0} \\ 
\hat{m}_{1} \\ 
\\ 
\hat{m}_{k-1}
\end{array}  \right]$$
We conclude the proof using the Cramer's rule which provides explicit formula for the solution of a system of linear equations with as many equations as unknowns.
\end{proof}

Then returning to the definition of the discrete measure $\hat{\mu}$ we explicitly express $\hat{Q}_{k}(\lambda_{i})$ in terms of $(\lambda_{i})_{1\leq i \leq n}$ and $(u_{i}^{T}Y)_{1\leq i \leq n}$ for all $1\leq k < n$  and all $1\leq i \leq n$.
\begin{prop}
\label{prop: expression-det}
Let $k \in\llbracket 1,\dots,n-1 \rrbracket$ and $i \in\llbracket 1,\dots,n  \rrbracket$.

Let $\hat{p}_{i}:=Y^{T}u_{i}$. Define 
$$I_{k}=\left\lbrace (j_{1},...,j_{k})\in \llbracket1,n \rrbracket^{k},
j_{1}\neq ...\neq j_{k}\right\rbrace$$ and $$I_{k,i}=\left\lbrace (j_{1},...,j_{k})\in \llbracket 1,n \rrbracket^{k},
j_{1}\neq ...\neq j_{k}\neq i  \right\rbrace.$$
We have
\begin{equation}
\label{prop: expression-poly2}
\hat{Q}_{k}(\lambda_{i})=
(-1)^{k} \dfrac{\sum_{(j_{1},..,j_{k})\in I_{k,i}}\hat{p}_{j_{1}}^{2}...\hat{p}_{j_{k}}^{2}V(\lambda_{j_{1}},...,\lambda_{j_{k}},\lambda_{i})\lambda_{j_{1}}...\lambda_{j_{k}}^{k}}{\sum_{(j_{1},..,j_{k})\in I_{k}}\hat{p}_{j_{1}}^{2}...\hat{p}_{j_{k}}^{2}V(\lambda_{j_{1}},...,\lambda_{j_{k}})\lambda_{j_{1}}^{2}...\lambda_{j_{k}}^{k+1}}
\end{equation}

where $V(x_{1},...,x_{l})$ is the Vandermonde determinant of $(x_{1},...,x_{l}) \in \mathbb{R}^{l}$.

If $k=n$ we have
$$\hat{Q}_{k}(\lambda_{i})=0.$$

\end{prop}

\begin{proof}
Let $1\leq i\leq n$. Using the fact that $d\hat{\mu}=\sum_{j=1}^{n}\lambda_{j}\hat{p}_{j}^{2}\delta_{\lambda_{j}}$ we get
$$\textrm{det}\left[ \begin{array}{cccc}
\hat{m}_{0} & \hat{m}_{1} & ... & \hat{m}_{k} \\ 
\vdots &  &  &  \\ 
\hat{m}_{k-1} & \hat{m}_{k} &  & \hat{m}_{2k-1} \\ 
1 & \lambda_{i} & ... & \lambda_{i}^{k}
\end{array} \right]$$
$$=\textrm{det}\left[ \begin{array}{cccc}
\sum_{j=1}^{n}\lambda_{j}\hat{p}_{j}^{2} & \sum_{j=1}^{n}\lambda_{j}^{2}\hat{p}_{j}^{2} &...  & \sum_{j=1}^{n}\lambda_{j}^{k+1}\hat{p}_{j} ^{2}\\ 
\vdots &  &  &  \\ 
 \sum_{j=1}^{n}\lambda_{j}^{k}\hat{p}_{j}^{2}& \sum_{j=1}^{n}\lambda_{j}^{k}\hat{p}_{j}^{2}&  & \sum_{j=1}^{n}\lambda_{j}^{2k}\hat{p}_{j}^{2} \\ 
1 & \lambda_{i} & ... & \lambda_{i}^{k}
\end{array} \right]$$

$$=\sum_{j_{1}=1}^{n}...\sum_{j_{k}=1}^{n}\hat{p}_{j_{1}}^{2}\hat{p}_{j_{2}}^{2}...\hat{p}_{j_{k}}^{2}\lambda_{j_{1}}\lambda_{j_{2}}^{2}...\lambda_{j_{k}}^{k}\textrm{det}\left[ \begin{array}{cccc}
1 & \lambda_{j_{1}}& ... & \lambda_{j_{1}}^{k} \\ 
\vdots &  &  &  \\ 
1& \lambda_{j_{k}}&  & \lambda_{j_{k}}^{k} \\ 
1 & \lambda_{i} & ... & \lambda_{i}^{k}
\end{array} \right]$$
where $$\textrm{det}\left[ \begin{array}{cccc}
1 & \lambda_{j_{1}}&  & \lambda_{j_{1}}^{k} \\ 
\vdots &  &  &  \\ 
1& \lambda_{j_{k}}&  & \lambda_{j_{k}}^{k} \\ 
1 & \lambda_{i} & ... & \lambda_{i}^{k}
\end{array} \right]=V(\lambda_{j_{1}},...,\lambda_{j_{k}}, \lambda_{i}).$$ $V(\lambda_{j_{1}},...,\lambda_{j_{k}}, \lambda_{i})$ is the Vandermonde determinant of $\lambda_{j_{1}},...,\lambda_{j_{k}}, \lambda_{i}$ and is non zero only if all the $\lambda_{j_{1}},...,\lambda_{j_{k}}, \lambda_{i}$ are distincts.

Therefore if $k<n$, we get 
\begin{equation}
\label{eq: det1}
\textrm{det}\left[ \begin{array}{cccc}
\hat{m}_{0} & \hat{m}_{1} & ... & \hat{m}_{k} \\ 
 \vdots&  &  &  \\ 
\hat{m}_{k-1} & \hat{m}_{k} &  & \hat{m}_{2k-1} \\ 
1 & \lambda_{i} & ... & \lambda_{i}^{k}
\end{array} \right]=\sum_{(j_{1},..,j_{k})\in I_{k,i}}\hat{p}_{j_{1}}^{2}\hat{p}_{j_{2}}^{2}...\hat{p}_{j_{k}}^{2}\lambda_{j_{1}}\lambda_{j_{2}}^{2}...\lambda_{j_{k}}^{k} V(\lambda_{j_{1}},...,\lambda_{j_{k}}, \lambda_{i}).
\end{equation}
Using the same arguments we also get
\begin{equation}
\label{eq: det2}
\textrm{det}\left[ \begin{array}{cccc}
\hat{m}_{1} & \hat{m}_{2} & ...  & \hat{m}_{k} \\ 
\vdots &  &  &  \\ 
\hat{m}_{k-1} & \hat{m}_{k} &  & \hat{m}_{2k-2} \\ 
 \hat{m}_{k}&  \hat{m}_{k+1} & ... &  \hat{m}_{2k-1}
\end{array} \right]= \sum_{(j_{1},..,j_{k})\in I_{k}} \hat{p}_{j_{1}}^{2}\hat{p}_{j_{2}}^{2}...\hat{p}_{j_{k}}^{2}\lambda_{j_{1}}^{2}\lambda_{j_{2}}^{3}...\lambda_{j_{k}}^{k+1}V(\lambda_{j_{1}},...,\lambda_{j_{k}}). 
\end{equation}
From (\ref{prop: expression-poly}), (\ref{eq: det1}) and (\ref{eq: det2}) we deduce (\ref{prop: expression-poly2}).

When $k=n$,  $\textrm{det}\left[ \begin{array}{cccc}
\hat{m}_{0} & \hat{m}_{1} & ... & \hat{m}_{k} \\ 
\vdots &  &  &  \\ 
\hat{m}_{k-1} & \hat{m}_{k} &  & \hat{m}_{2k-1} \\ 
1 & \lambda_{i} & ... & \lambda_{i}^{k}
\end{array} \right]=0$
and therefore $$\hat{Q}_{k}(\lambda_{i})=0.$$
\end{proof}

Now using the properties of the Vandermonde determinant we provide a more useful characterization of the residual $\hat{Q}_{k}(\lambda_{i})$.
Let $k<n$. Formula (\ref{prop: expression-poly2}) of Proposition \ref{prop: expression-det} tells us that
\begin{equation}
\label{eq: preuve-prop-exp-det1}
\hat{Q}_{k}(\lambda_{i})=
(-1)^{k} \dfrac{\sum_{(j_{1},..,j_{k})\in I_{k,i}}\hat{p}_{j_{1}}^{2}...\hat{p}_{j_{k}}^{2}V(\lambda_{j_{1}},...,\lambda_{j_{k}},\lambda_{i})\lambda_{j_{1}}...\lambda_{j_{k}}^{k}}{\sum_{(j_{1},..,j_{k})\in I_{k}}\hat{p}_{j_{1}}^{2}...\hat{p}_{j_{k}}^{2}V(\lambda_{j_{1}},...,\lambda_{j_{k}})\lambda_{j_{1}}^{2}...\lambda_{j_{k}}^{k+1}}.
\end{equation}

On the one hand, we have
$$\sum_{(j_{1},..,j_{k})\in I_{k}}\hat{p}_{j_{1}}^{2}...\hat{p}_{j_{k}}^{2}V(\lambda_{j_{1}},...,\lambda_{j_{k}})\lambda_{j_{1}}^{2}...\lambda_{j_{k}}^{k+1}$$
$$ =\sum_{(j_{1},..,j_{k})\in I^{+}_{k}} \sum_{\tau \in \mathcal{S}(1,...,k)}\hat{p}_{j_{\tau (1)}}^{2}...\hat{p}_{j_{\tau(k)}}^{2}V(\lambda_{j_{\tau(1)}},...,\lambda_{j_{\tau(k)}})\lambda_{j_{\tau(1)}}^{2}...\lambda_{j_{\tau(k)}}^{k+1}$$
where $ \mathcal{S}(1,...,k)$ is the set formed of all the permutations of $(1,...,k)$.
Then using the fact that $V(\lambda_{j_{\tau(1)}},...,\lambda_{j_{\tau(k)}})=\varepsilon(\tau)V(\lambda_{j_{1}},...,\lambda_{j_{k}})$ we get
$$\sum_{(j_{1},..,j_{k})\in I_{k}}\hat{p}_{j_{1}}^{2}...\hat{p}_{j_{k}}^{2}V(\lambda_{j_{1}},...,\lambda_{j_{k}})\lambda_{j_{1}}^{2}...\lambda_{j_{k}}^{k+1}$$
$$=\sum_{(j_{1},..,j_{k})\in I^{+}_{k}} \sum_{\tau \in \mathcal{S}(1,...,k)}\hat{p}_{j_{1}}^{2}...\hat{p}_{j_{k}}^{2}\varepsilon(\tau)V(\lambda_{j_{1}},...,\lambda_{j_{k}})\lambda_{j_{1}}^{2}...\lambda_{j_{k}}^{2}\lambda_{j_{\tau(2)}}...\lambda_{j_{\tau(k)}}^{k-1}$$

\begin{equation}
\label{eq: expression-residual1}
=\sum_{(j_{1},..,j_{k})\in I^{+}_{k}} \hat{p}_{j_{1}}^{2}...\hat{p}_{j_{k}}^{2}V(\lambda_{j_{1}},...,\lambda_{j_{k}})\lambda_{j_{1}}^{2}...\lambda_{j_{k}}^{2}\left[ \sum_{\tau \in \mathcal{S}(1,...,k)}\varepsilon(\tau)\lambda_{j_{\tau(2)}}...\lambda_{j_{\tau(k)}}^{k-1}\right].
\end{equation}
On the other hand, 
\begin{equation}
\label{eq: expression-residual2}
 V(\lambda_{j_{1}},...,\lambda_{j_{k}})=\sum_{\tau \in \mathcal{S}(1,...,k)}\varepsilon(\tau)\lambda_{j_{1}}^{\tau(1)-1}...\lambda_{j_{k}}^{\tau(k)-1}=\sum_{\tau \in \mathcal{S}(1,...,k)}\varepsilon(\tau)\lambda_{j_{\tau(2)}}...\lambda_{j_{\tau(k)}}^{k-1}.
\end{equation}

To conclude (\ref{eq: expression-residual1}) and (\ref{eq: expression-residual2}) leads to

\begin{equation}
\label{eq: preuve-prop-exp-det2}
\sum_{(j_{1},..,j_{k})\in I_{k}}\hat{p}_{j_{1}}^{2}...\hat{p}_{j_{k}}^{2}V(\lambda_{j_{1}},...,\lambda_{j_{k}})\lambda_{j_{1}}^{2}...\lambda_{j_{k}}^{k+1}=\sum_{(j_{1},..,j_{k})\in I^{+}_{k}} \hat{p}_{j_{1}}^{2}...\hat{p}_{j_{k}}^{2}\lambda_{j_{1}}^{2}...\lambda_{j_{k}}^{2}V(\lambda_{j_{1}},...,\lambda_{j_{k}})^{2}.
\end{equation}

A similar reasoning can be applied to the numerator. Indeed using the fact that  $$V(\lambda_{j_{1}},...,\lambda_{j_{k}},\lambda_{i})= \prod_{l=1}^{k}(\lambda_{i}-\lambda_{j_{l}})\prod_{1\leq q < m\leq k} (\lambda_{j_{m}}-\lambda_{j_q})=\prod_{l=1}^{k}(\lambda_{i}-\lambda_{j_{l}})V(\lambda_{j_{1}},...,\lambda_{j_{k}})$$ we get
$$\sum_{(j_{1},..,j_{k})\in I_{k,i}}\hat{p}_{j_{1}}^{2}...\hat{p}_{j_{k}}^{2}V(\lambda_{j_{1}},...,\lambda_{j_{k}},\lambda_{i})\lambda_{j_{1}}...\lambda_{j_{k}}^{k}$$
$$=\sum_{(j_{1},..,j_{k})\in I^{+}_{k,i}}\sum_{\tau \in \mathcal{S}(1,...,k)}\hat{p}_{j_{\tau (1)}}^{2}...\hat{p}_{j_{\tau(k)}}^{2}\prod_{l=1}^{k}(\lambda_{i}-\lambda_{j_{\tau(l)}})V(\lambda_{j_{\tau(1)}},...,\lambda_{j_{\tau(k)}})\lambda_{j_{\tau(1)}}...\lambda_{j_{\tau(k)}}^{k}$$
$$=\sum_{(j_{1},..,j_{k})\in I^{+}_{k,i}}\hat{p}_{j_{1}}^{2}...\hat{p}_{j_{k}}^{2}\prod_{l=1}^{k}(\lambda_{i}-\lambda_{j_{l}})V(\lambda_{j_{1}},...,\lambda_{j_{k}})\lambda_{j_{1}}...\lambda_{j_{k}}\left[ \sum_{\tau \in \mathcal{S}(1,...,k)}\varepsilon(\tau)\lambda_{j_{\tau(2)}}...\lambda_{j_{\tau(k)}}^{k-1}\right] $$
$$
=\sum_{(j_{1},..,j_{k})\in I^{+}_{k}}\hat{p}_{j_{1}}^{2}...\hat{p}_{j_{k}}^{2}\lambda_{j_{1}}...\lambda_{j_{k}}V(\lambda_{j_{1}},...,\lambda_{j_{k}})^{2}\prod_{l=1}^{k}(\lambda_{i}-\lambda_{j_{l}})
$$
\begin{equation}
\label{eq: preuve-prop-exp-det3}
= (-1)^k\sum_{(j_{1},..,j_{k})\in I^{+}_{k}}\hat{p}_{j_{k}}^{2}\lambda_{j_{1}}^2...\lambda_{j_{k}}^2V(\lambda_{j_{1}},...,\lambda_{j_{k}})^{2}\prod_{l=1}^{k}(1-\frac{\lambda_{i}}{\lambda_{j_{l}}}).
\end{equation}

From (\ref{eq: preuve-prop-exp-det1}), (\ref{eq: preuve-prop-exp-det2}) and (\ref{eq: preuve-prop-exp-det3}) we conclude

$$\hat{Q}_{k}(\lambda_{i})=\sum_{(j_{1},..,j_{k})\in I^{+}_{k}}\left[ \dfrac{\hat{p}_{j_{1}}^{2}...\hat{p}_{j_{k}}^{2}\lambda_{j_{1}}^{2}...\lambda_{j_{k}}^{2}V(\lambda_{j_{1}},...,\lambda_{j_{k}})^{2}}{\sum_{(j_{1},..,j_{k})\in I^{+}_{k}} \hat{p}_{j_{1}}^{2}...\hat{p}_{j_{k}}^{2}\lambda_{j_{1}}^{2}...\lambda_{j_{k}}^{2}V(\lambda_{j_{1}},...,\lambda_{j_{k}})^{2}}\right] \prod_{l=1}^{k}(1-\frac{\lambda_{i}}{\lambda_{j_{l}}}).$$

\subsection{Proof of Proposition \ref{prop:empirical-risk}}
\begin{proof}
Let $k<n$. By definition of $\hat{\beta}_{k}$ and referring to results of Proposition \ref{prop:poly} we have
$$\Vert Y-X\hat{\beta}_{k}\Vert^{2}=\underset{Q \in \mathcal{P}_{k,1}}{\rm{min}} \Vert Q(XX^{T})Y\Vert^{2}=\underset{Q \in \mathcal{P}_{k,1}}{\rm{min}} \Vert Q(XX^{T})\left( X\beta^{*}+\varepsilon\right) \Vert^{2}.$$
Using the decomposition of $\beta^{*}$ and $\varepsilon$ on the left and right eigenvectors (i.e. $\beta^{*}=\sum_{i=1}^{p}\tilde{\beta}_{i}^{*}v_{i}$ where $\tilde{\beta}_{i}^{*}=\beta^{T}v_{i}$ and $\varepsilon=\sum_{i=1}^{n}\tilde{\varepsilon}_{i}u_{i}$ where $\tilde{\varepsilon}_{i}=\varepsilon^{T}u_{i}$) we get
$$\Vert Y-X\hat{\beta}_{k}\Vert^{2}= \underset{Q \in \mathcal{P}_{k,1}}{\rm{min}} \left(  \sum_{i =1 }^{n}Q(\lambda_{i})^{2}\left( \sqrt{\lambda_{i}}\tilde{\beta}_{i}^{*}+ \tilde{\varepsilon}_{i}\right)^{2} \right)$$
$$\leqslant \left( \underset{Q \in \mathcal{P}_{k,1}}{\rm{min}}\underset{\lambda \in \left[\lambda_{n}, \lambda_{1} \right] }{\rm{max}}Q(\lambda)^{2}\right)   \sum_{i=1}^{n}\left( \sqrt{\lambda_{i}}\tilde{\beta}_{i}^{*}+ \tilde{\varepsilon}_{i}\right)^{2}.$$
Then we have
$$\Vert Y-X\hat{\beta}_{k}\Vert^{2}\leq  \left( \underset{Q \in \mathcal{P}_{k,1} }{\rm{min}}\underset{\lambda \in \left[\lambda_{n}, \lambda_{1} \right] }{\rm{max}}\mid Q(\lambda)\mid\right) ^{2}\sum_{i =1}^{n}\left( \sqrt{\lambda_{i}}\tilde{\beta}_{i}^{*}+ \tilde{\varepsilon}_{i}\right)^{2}$$
$$\leqslant \dfrac{1}{\left[ C_{k}\left(\dfrac{\lambda_1+ \lambda_n}{\lambda_1-\lambda_n}  \right)\right]^{2} } \sum_{i =1}^{n}\left( \sqrt{\lambda_{i}}\tilde{\beta}_{i}^{*}+ \tilde{\varepsilon}_{i}\right)^{2} $$
where $C_{k}$ is the $k^{th}$ Chebyschev polynomial. This last inequalities follows from Proposition \ref{theo:chebychev}.
Then we use the fact that
 $$\left\lvert C_{k}\left( \frac{\lambda_{1}+\lambda_{n}}{\lambda_{1}-\lambda_{n}}\right) \right\rvert=\frac{1}{2}\left\lvert \left(\frac{\sqrt{\lambda}_{1}+\sqrt{\lambda}_{n}}{\sqrt{\lambda}_{1}-\sqrt{\lambda}_{n}} \right) ^{k}+\left(\frac{\sqrt{\lambda}_{1}-\sqrt{\lambda}_{n}}{\sqrt{\lambda}_{1}+\sqrt{\lambda}_{n}} \right) ^{k}\right\rvert$$
 $$= \frac{1}{2}\left\lvert \left(\frac{\sqrt{C(X^{T}X)}+1}{\sqrt{C(X^{T}X)}-1} \right) ^{k}+\left(\frac{\sqrt{C(X^{T}X)}-1}{\sqrt{C(X^{T}X)}+1} \right) ^{k}\right\rvert\geqslant \left(\frac{\sqrt{C(X^{T}X)}+1}{\sqrt{C(X^{T}X)}-1} \right) ^{k},$$
where $C(X^{T}X)=\frac{\lambda_{1}}{\lambda_{n}}$.
At last we get
$$\mathbb{E}\left[  \frac{1}{n}\Vert Y-X\hat{\beta}_{k}\Vert^{2}\right]  \leqslant \frac{1}{n}\left(\frac{\sqrt{C(X^{T}X)}-1}{\sqrt{C(X^{T}X)}+1} \right) ^{2k}\mathbb{E}\left[ \sum_{i=1}^{n}\left( \sqrt{\lambda_{i}}\tilde{\beta}_{i}^{*}+ \tilde{\varepsilon}_{i}\right)^{2} \right]  $$
and since the $(\varepsilon_{j})_{1\leq j \leq n}$ are assumed to be centered we conclude
$$\mathbb{E}\left[  \frac{1}{n}\Vert Y-X\hat{\beta}_{k}\Vert^{2}\right] \leqslant \left(\frac{\sqrt{C(X^{T}X)}-1}{\sqrt{C(X^{T}X)}+1} \right) ^{2k}\left[ \frac{1}{n}\parallel X\beta^{*}\parallel^{2}+ \sigma^{2}\right]  .$$
\end{proof}

\subsection{Proof of Proposition \ref{prop: first term}}
\begin{proof}
We have
$$\frac{1}{n}\parallel X\beta^{*}-XP^{*}_{k}(X^{T}X)X^{T}Y\parallel^{2}$$
$$\leq \frac{2}{n}\parallel X\beta^{*}-XP^{*}_{k}(X^{T}X)X^{T}X^{T}X\beta^{*}\parallel^{2}+\frac{2}{n}\parallel XP^{*}_{k}(X^{T}X)X^{T}\varepsilon\parallel^{2} $$
$$= \frac{2}{n}\parallel Q^{*}_{k}(XX^{T})X\beta^{*}\parallel^{2}+\frac{2}{n}\parallel XP^{*}_{k}(X^{T}X)X^{T}\varepsilon\parallel^{2}.$$

On one hand, by the same arguments as the ones used to prove Proposition \ref{prop:empirical-risk} (with no noise), we get
\begin{equation}
\label{eq:first term}
\frac{1}{n}\parallel Q^{*}_{k}(XX^{T})X\beta^{*}\parallel^{2}\leq 
\frac{1}{n}\left(\frac{\sqrt{C(X^{T}X)}-1}{\sqrt{C(X^{T}X)}+1} \right) ^{k} \Vert X\beta^{*}\Vert^{2}
\end{equation}
where $C(X^{T}X)=\frac{\lambda_{1}}{\lambda_{n}}$ is the ratio of the two extreme non zero eigenvalues of $X^{T}X$.

On the other hand we have
$$\frac{1}{n}\parallel XP^{*}_{k}(X^{T}X)X^{T}\varepsilon\parallel^{2}= \frac{1}{n}\sum_{i=1}^{n}\left(1-Q^{*}_{k}(\lambda_{i}) \right)^{2}\tilde{\varepsilon}_{i}^{2}$$
$$= \frac{1}{n}\sum_{i=1}^{n}\left(1-Q^{*}_{k}(\lambda_{i}) \right)^{2}p_{i}^2 \frac{\tilde{\varepsilon}_{i}^{2} }{p_{i}^2} $$
where $\tilde{\varepsilon}_{i}=\varepsilon^{T}u_{i}$. 
Notice that $Q^{*}_{k}(\lambda_{i})$ can be positive or negative and therefore the factors in the last sum oscillate above and below one (see Subsection \ref{subsection:shrink}). We bound this last term by above using concentration inequalities. Here a low variance of the noise is necessary to ensure that the term we consider is not too large.
The random variables $(\varepsilon_{i})_{1\leq i\leq n}$ are assumed to be i.i.d $\sim \mathcal{N}(0,\sigma_{n}^{2})$ and so are the $(\tilde{\varepsilon}_{i})_{1\leq i\leq n}$.
Therefore we use the following proposition which is a direct consequence of concentration inequalities for Gaussian random variables

\begin{prop}
\label{prop:concentration-variance}
Let $\mathcal{A}=\left\lbrace \cap_{i=1}^{n}\left \lvert \tilde{\varepsilon_{i}}\right \rvert \leq \delta \right\rbrace $.
If assumptions (H.1) holds then there exists a constant $C>1$ such that
$$\mathbb{P}(\mathcal{A}^{c})\leq \sum_{i=1}^{n}\mathbb{P}(\mid \tilde{\varepsilon}_{i}\mid > \delta )\leq \sum_{i=1}^{n} e^{-\frac{\delta^{2}}{2\sigma_{n}^{2}}}\leq ne^{-C\delta^{2}n}.$$ 
In addition with probability at least $1-n^{1-C}$ we have for all $i=1,...,n$
$$\left  \lvert \tilde{\varepsilon_{i}}\right \rvert \leq \sqrt{\frac{\log(n)}{n}}$$
\end{prop}

With Proposition \ref{prop:concentration-variance} we deduce that with probability at least $1-n^{1-C}$ where $C>1$ we have for all $i=1,...,n$
 $$\frac{1}{n}\parallel XP^{*}_{k}(X^{T}X)X^{T}\varepsilon\parallel^{2}\leq  \frac{1}{n}\left(\sum_{i=1}^{n}\left(1-Q^{*}_{k}(\lambda_{i}) \right)^{2}p_{i}^{2}\right) \frac{\log(n)}{nL}. $$

Then using the triangular inequality and (\ref{eq:first term}) we state
\begin{equation}
\label{eq:Second term}
\frac{1}{n}\parallel XP^{*}_{k}(X^{T}X)X^{T}\varepsilon\parallel^{2}\leq \frac{1}{n}\left[2+2  \left(\frac{\sqrt{C(X^{T}X)}-1}{\sqrt{C(X^{T}X)}+1} \right) ^{2k} \right] \Vert X\beta^{*}\Vert^{2}\frac{\log(n)}{nL}.
\end{equation}

Combining (\ref{eq:first term}) and (\ref{eq:Second term}) we conclude
$$\frac{1}{n}\parallel X\beta^{*}-XP^{*}_{k}(X^{T}X)X^{T}Y\parallel^{2} $$
\begin{equation}
\label{eq:First-term-final}
\leq \frac{1}{n}\left[ 2\left(\frac{\sqrt{C(X^{T}X)}-1}{\sqrt{C(X^{T}X)}+1} \right) ^{2k}+4\dfrac{\log(n)}{nL}\left(1+ \left(  \frac{\sqrt{C(X^{T}X)}-1}{\sqrt{C(X^{T}X)}+1}\right) ^{2k} \right)  \right]  \Vert X\beta^{*}\Vert^{2}.
\end{equation}

\end{proof}

\subsection{Proof of Proposition \ref{prop:pred-second-term}}
\begin{proof}
Using the SVD of $XX^{T}$ we get
\begin{equation}
\label{eq:secondterm-SVD}
\frac{1}{n}\parallel\left( \hat{Q}_{k}(X^{T}X)-Q^{*}_{k}(X^{T}X)\right) Y\parallel^{2}=\frac{1}{n}\sum_{i=1}^{n}\left( \hat{Q}_{k}(\lambda_{i})-Q^{*}_{k}(\lambda_{i})\right) ^{2}\hat{p}_{i}^{2}.
\end{equation}

\begin{rmq} 
We can notice that
$$\sum_{i=1}^{n}\left( \hat{Q}_{k}(\lambda_{i})-Q^{*}_{k}(\lambda_{i})\right) ^{2}\hat{p}_{i}^{2}\leq \frac{1}{\lambda_{n}}\sum_{i=1}^{n}\left( \hat{Q}_{k}(\lambda_{i})-Q^{*}_{k}(\lambda_{i})\right) ^{2}\lambda_{i}\hat{p}_{i}^{2}\leq  \frac{1}{\lambda_{n}}\parallel \hat{Q}_{k}-Q^{*}_{k}\parallel_{\hat{\mu }}^{2}.$$
\end{rmq}

We define
$$\hat{D}_{j_{1},..,j_{k}}:=\hat{p}_{j_{1}}^{2}...\hat{p}_{j_{k}}^{2}\lambda_{j_{1}}^{2}...\lambda_{j_{k}}^{2}V(\lambda_{j_{1}},...,\lambda_{j_{k}})^{2}>0,$$

$$D_{j_{1},..,j_{k}}:=p_{j_{1}}^{2}...p_{j_{k}}^{2}\lambda_{j_{1}}^{2}...\lambda_{j_{k}}^{2}V(\lambda_{j_{1}},...,\lambda_{j_{k}})^{2}>0.$$
and 
$$\hat{D}_{k}:=\sum_{(j_{1},..,j_{k})\in I^{+}_{k}}\hat{D}_{j_{1},..,j_{k}}$$,
$$D_{k}:=\sum_{(j_{1},..,j_{k})\in I^{+}_{k}}D_{j_{1},..,j_{k}}.$$
We recall that
$$\hat{Q}_{k}(\lambda_{i})=
(-1)^{k}\dfrac{\sum_{(j_{1},..,j_{k})\in I^{+}_{k}} \hat{D}_{j_{1},..,j_{k}}\prod_{l=1}^{k}(\frac{\lambda_{i}}{\lambda_{j_{l}}}-1) }{\sum_{(j_{1},..,j_{k})\in I^{+}_{k}}\hat{D}_{j_{1},..,j_{k}}}$$
and
$$Q_{k}(\lambda_{i})=
(-1)^{k}\dfrac{\sum_{(j_{1},..,j_{k})\in I^{+}_{k}} D_{j_{1},..,j_{k}}\prod_{l=1}^{k}(\frac{\lambda_{i}}{\lambda_{j_{l}}}-1) }{\sum_{(j_{1},..,j_{k})\in I^{+}_{k}}D_{j_{1},..,j_{k}}}.$$
We have
$$\left \lvert \hat{Q}_{k}(\lambda_{i})- Q_{k}^{*}(\lambda_{i}) \right \rvert \leq \left \lvert \dfrac{\sum_{(j_{1},..,j_{k})\in I^{+}_{k}}\left[ \hat{D}_{j_{1},..,j_{k}}\prod_{l=1}^{k}(\frac{\lambda_{i}}{\lambda_{j_{l}}}-1)\right] }{\hat{D}_{k}}-\dfrac{\sum_{(j_{1},..,j_{k})\in I^{+}_{k}}\left[ D_{j_{1},..,j_{k}}\prod_{l=1}^{k}(\frac{\lambda_{i}}{\lambda_{j_{l}}}-1)\right] }{D_{k}} \right \rvert$$
$$\leq \left \lvert \dfrac{\sum_{(j_{1},..,j_{k})\in I^{+}_{k}}\left[ D_{j_{1},..,j_{k}}\prod_{l=1}^{k}(\frac{\lambda_{i}}{\lambda_{j_{l}}}-1)\right] }{D_{k}}-\dfrac{\sum_{(j_{1},..,j_{k})\in I^{+}_{k}}\left[ \hat{D}_{j_{1},..,j_{k}}\prod_{l=1}^{k}(\frac{\lambda_{i}}{\lambda_{j_{l}}}-1)\right] }{D_{k}}  \right \rvert$$
$$+
\left \lvert  \dfrac{\sum_{(j_{1},..,j_{k})\in I^{+}_{k}}\left[ \hat{D}_{j_{1},..,j_{k}}\prod_{l=1}^{k}(\frac{\lambda_{i}}{\lambda_{j_{l}}}-1)\right] }{D_{k}}-\dfrac{\sum_{(j_{1},..,j_{k})\in I^{+}_{k}}\left[ \hat{D}_{j_{1},..,j_{k}}\prod_{l=1}^{k}(\frac{\lambda_{i}}{\lambda_{j_{l}}}-1)\right] }{\hat{D}_{k}}  \right \rvert$$

$$\leq \dfrac{1}{D_{k}}\left \lvert  \sum_{(j_{1},..,j_{k})\in I^{+}_{k}}\left[ D_{j_{1},..,j_{k}}-\hat{D}_{j_{1},..,j_{k}} \right] \prod_{l=1}^{k}(\frac{\lambda_{i}}{\lambda_{j_{l}}}-1) \right \rvert$$
$$+ \dfrac{1}{\left( D_{k}\hat{D}_{k}\right) }\left \lvert \sum_{(j_{1},..,j_{k})\in I^{+}_{k}}\hat{D}_{j_{1},..,j_{k}} \prod_{l=1}^{k}(\frac{\lambda_{i}}{\lambda_{j_{l}}}-1) \right \rvert \left \lvert  \sum_{(j_{1},..,j_{k})\in I^{+}_{k}}\left[ D_{j_{1},..,j_{k}}-\hat{D}_{j_{1},..,j_{k}} \right]\right \rvert $$

Besides we have $$\left \lvert \sum_{(j_{1},..,j_{k})\in I^{+}_{k}}\hat{D}_{j_{1},..,j_{k}} \prod_{l=1}^{k}(\frac{\lambda_{i}}{\lambda_{j_{l}}}-1) \right \rvert \leq  \sum_{(j_{1},..,j_{k})\in I^{+}_{k}}\hat{D}_{j_{1},..,j_{k}} \left[  \underset{I_{k}^{+}}{\rm{max}}\left( \prod_{l=1}^{k}\left \lvert\frac{\lambda_{i}}{\lambda_{j_{l}}}-1\right \rvert\right)\right]   , $$

$$\left \lvert  \sum_{(j_{1},..,j_{k})\in I^{+}_{k}}\left[ D_{j_{1},..,j_{k}}-\hat{D}_{j_{1},..,j_{k}} \right] \prod_{l=1}^{k}(\frac{\lambda_{i}}{\lambda_{j_{l}}}-1) \right \rvert =\left \lvert  \sum_{(j_{1},..,j_{k})\in I^{+}_{k}} D_{j_{1},..,j_{k}}\left(1-\frac{\hat{p}_{j_{1}}^{2}...\hat{p}_{j_{k}}^{2}}{p_{j_{1}}^{2}...p_{j_{k}}^{2}} \right) \prod_{l=1}^{k}(\frac{\lambda_{i}}{\lambda_{j_{l}}}-1) \right \rvert$$
$$\leq \sum_{(j_{1},..,j_{k})\in I^{+}_{k}} D_{j_{1},..,j_{k}}\left[  \underset{I_{k}^{+}}{\rm{max}}\left( 1-\frac{\hat{p}_{j_{1}}^{2}...\hat{p}_{j_{k}}^{2}}{p_{j_{1}}^{2}...p_{j_{k}}^{2}} \right)\right]  \left[ \underset{I_{k}^{+}}{\rm{max}}\left( \prod_{l=1}^{k}\left \lvert \frac{\lambda_{i}}{\lambda_{j_{l}}}-1\right \rvert\right)\right]  $$
and $$\left \lvert  \sum_{(j_{1},..,j_{k})\in I^{+}_{k}}\left[ D_{j_{1},..,j_{k}}-\hat{D}_{j_{1},..,j_{k}} \right]=  \right \rvert\left \lvert  \sum_{(j_{1},..,j_{k})\in I^{+}_{k}} D_{j_{1},..,j_{k}}\left(1-\frac{\hat{p}_{j_{1}}^{2}...\hat{p}_{j_{k}}^{2}}{p_{j_{1}}^{2}...p_{j_{k}}^{2}} \right) \right \rvert $$
$$\leq \sum_{(j_{1},..,j_{k})\in I^{+}_{k}} D_{j_{1},..,j_{k}}\left[  \underset{I_{k}^{+}}{\rm{max}}\left( 1-\frac{\hat{p}_{j_{1}}^{2}...\hat{p}_{j_{k}}^{2}}{p_{j_{1}}^{2}...p_{j_{k}}^{2}} \right)\right]. $$

Therefore we get
\begin{equation}
\label{eq:residual-difference-1}
\left \lvert \hat{Q}_{k}(\lambda_{i})- Q_{k}^{*}(\lambda_{i}) \right \rvert \leq  2\left[  \underset{I_{k}^{+}}{\textrm{max}}\left( 1-\frac{\hat{p}_{j_{1}}^{2}...\hat{p}_{j_{k}}^{2}}{p_{j_{1}}^{2}...p_{j_{k}}^{2}} \right)\right]\left[ \underset{I_{k}^{+}}{\rm{max}}\left( \prod_{l=1}^{k}\left \lvert \frac{\lambda_{i}}{\lambda_{j_{l}}}-1\right \rvert\right)\right] 
\end{equation}
where $$\frac{\hat{p}_{j_{1}}^{2}...\hat{p}_{j_{k}}^{2}}{p_{j_{1}}^{2}...p_{j_{k}}^{2}}=\left(1+\frac{\varepsilon_{j_{1}}}{p_{j_{1}}} \right)^{2}...\left(1+\frac{\varepsilon_{j_{k}}}{p_{j_{k}}} \right)^{2} .$$
From Proposition \ref{prop:concentration-variance} and (H.2) we have that there exists a constant $C>1$ such that 
\begin{equation}
\label{eq:residual-difference-2}
\left( 1-\frac{C}{\sqrt{L}}\sqrt{\frac{\log n}{n}}\right)^{2k}\leq\frac{\hat{p}_{j_{1}}^{2}...\hat{p}_{j_{k}}^{2}}{p_{j_{1}}^{2}...p_{j_{k}}^{2}}\leq \left( 1+\frac{C}{\sqrt{L}}\sqrt{\frac{\log n}{n}}\right)^{2k}
\end{equation}
with probablity at least $1-n^{1-C}$.

From (\ref{eq:residual-difference-1}) and (\ref{eq:residual-difference-2}) we deduce that there exists a constant $\tilde{C}$ such that with probablity at least $1-n^{1-C}$ where $C>1$, 
\begin{equation}
\label{eq:residual-difference-3}
\left \lvert \hat{Q}_{k}(\lambda_{i})- Q_{k}^{*}(\lambda_{i}) \right \rvert \leq \frac{2k\tilde{C}}{\sqrt{L}}\sqrt{\frac{\log n}{n}}\left[ \underset{I_{k}^{+}}{\textrm{max}}\left( \prod_{l=1}^{k}\left \lvert \frac{\lambda_{i}}{\lambda_{j_{l}}}-1\right \rvert\right)\right] . 
\end{equation}
Finally using again Proposition \ref{prop:concentration-variance} and (\ref{eq:residual-difference-3}) we get

$$\frac{1}{n}\parallel\left( \hat{Q}_{k}(X^{T}X)-Q^{*}_{k}(X^{T}X)\right) Y\parallel^{2}=\frac{1}{n}\sum_{i=1}^{n}\left( \hat{Q}_{k}(\lambda_{i})-Q^{*}_{k}(\lambda_{i})\right) ^{2}\dfrac{\hat{p}_{i}^{2}}{p_{i}^{2}}p_{i}^{2}$$
$$\leq \frac{4k^{2}\tilde{C}^{2}}{L}\frac{\log n}{n^{2}}\left(1+\frac{C}{\sqrt{L}}\sqrt{\frac{\log n}{n}} \right)\sum_{i=1}^{n}\left[ \underset{I_{k}^{+}}{\textrm{max}}\left( \prod_{l=1}^{k}\left \lvert \frac{\lambda_{i}}{\lambda_{j_{l}}}-1\right \rvert\right)^{2}p_{i}^{2}\right]$$
where
$$\sum_{i=1}^{n}\left[ \underset{I_{k}^{+}}{\textrm{max}}\left( \prod_{l=1}^{k}\left \lvert \frac{\lambda_{i}}{\lambda_{j_{l}}}-1\right \rvert\right)^{2}p_{i}^{2}\right]=\Vert X\beta^{*}\Vert^{2}_{W}$$
 with $ W=\underset{1\leq i \leq n}{\textrm{diag}}\left( \underset{I_{k}^{+}}{\rm{max}}\left( \prod_{l=1}^{k}\left \lvert \frac{\lambda_{i}}{\lambda_{j_{l}}}-1\right \rvert^{2}\right) \right).$

\begin{rmq}

We can state sharper bounds for the ratio $\dfrac{\hat{Q}_{k}(\lambda_{i})}{Q_{k}(\lambda_{i })}$ for $i=1$ and $i=n$. 
Indeed we have 
$$\dfrac{\hat{Q}_{k}(\lambda_{i})}{Q^{*}_{k}(\lambda_{i})}=\dfrac{\sum_{(j_{1},..,j_{k})\in I_{k}^{+}}  \hat{D}_{j_{1},..,j_{k}}\prod_{l=1}^{k}(\frac{\lambda_{i}}{\lambda_{j_{l}}}-1)}{\sum_{(j_{1},..,j_{k})\in I_{k}^{+}} D_{j_{1},..,j_{k}}\prod_{l=1}^{k}(\frac{\lambda_{i}}{\lambda_{j_{l}}}-1)}\dfrac{\sum_{(j_{1},..,j_{k})\in I_{k}^{+}} D_{j_{1},..,j_{k}}}{\sum_{(j_{1},..,j_{k})\in I_{k}^{+}} \hat{D}_{j_{1},..,j_{k}}}$$

\begin{equation}
\label{eq:product}
= \dfrac{\sum_{(j_{1},..,j_{k})\in I_{k}^{+}} D_{j_{1},..,j_{k}}\prod_{l=1}^{k}(\frac{\lambda_{i}}{\lambda_{j_{l}}}-1)\left( \frac{\hat{p}_{j_{1}}^{2}...\hat{p}_{j_{k}}^{2}}{p_{j_{1}}^{2}...p_{j_{k}}^{2}}\right) }{\sum_{(j_{1},..,j_{k})\in I_{k}^{+}} D_{j_{1},..,j_{k}}\prod_{l=1}^{k}(\frac{\lambda_{i}}{\lambda_{j_{l}}}-1)}\dfrac{\sum_{(j_{1},..,j_{k})\in I_{k}^{+}} D_{j_{1},..,j_{k}}}{\sum_{(j_{1},..,j_{k})\in I_{k}^{+}} D_{j_{1},..,j_{k}}\left( \frac{\hat{p}_{j_{1}}^{2}...\hat{p}_{j_{k}}^{2}}{p_{j_{1}}^{2}...p_{j_{k}}^{2}}\right) }
\end{equation}

We are going to use again concentration inequalities to bound by above the two factors of the product in (\ref{eq:product}).
In fact on the event $\mathcal{A}$ we have (see (\ref{eq:residual-difference-2}))
$$\left( 1-C\sqrt{\frac{\log n}{nL}}\right)^{2k}\leq\frac{\hat{p}_{j_{1}}^{2}...\hat{p}_{j_{k}}^{2}}{p_{j_{1}}^{2}...p_{j_{k}}^{2}}.$$

Therefore, because all the terms $D_{j_{1},..,j_{k}}$ are positive, we deduce that
\begin{equation}
\label{eq:ratio-residual-1}
\dfrac{\sum_{(j_{1},..,j_{k})\in I_{k}^{+}} D_{j_{1},..,j_{k}}}{\sum_{(j_{1},..,j_{k})\in I_{k}} D_{j_{1},..,j_{k}}\left( \frac{\hat{p}_{j_{1}}^{2}...\hat{p}_{j_{k}}^{2}}{p_{j_{1}}^{2}...p_{j_{k}}^{2}}\right) }\leq \dfrac{1}{\left( 1-C\sqrt{\frac{\log n}{nL}}\right)^{2k}}.
\end{equation}
If we assume $C\sqrt{\frac{\log n}{nL}}<1$,  we get
\begin{equation}
\label{eq:second product}
\dfrac{\sum_{(j_{1},..,j_{k})\in I_{k}^{+}} D_{j_{1},..,j_{k}}}{\sum_{(j_{1},..,j_{k})\in I_{k}} D_{j_{1},..,j_{k}}\left( \frac{\hat{p}_{j_{1}}^{2}...\hat{p}_{j_{k}}^{2}}{p_{j_{1}}^{2}...p_{j_{k}}^{2}}\right) }\leq 1+O\left( 2k\sqrt{\frac{\log n}{nL}}\right) 
\end{equation}

Now we are going to bound by above the first factor in (\ref{eq:product}) for $i=1$ and $i=n$.
Let $i=1$. Then for all $(j_{1},..,j_{k})\in I_{k}^{+}$ we have 
$\prod_{l=1}^{k}(\frac{\lambda_{1}}{\lambda_{j_{l}}}-1)>0 $ and thus all the terms in the first factor are positive.
Therefore on the event $\mathcal{A}$ we get 
\begin{equation}
\label{eq:first product-1}
 \left \lvert\dfrac{\sum_{(j_{1},..,j_{k})\in I_{k}^{+}} D_{j_{1},..,j_{k}}\prod_{l=1}^{k}(\frac{\lambda_{1}}{\lambda_{j_{l}}}-1)\left( \frac{\hat{p}_{j_{1}}^{2}...\hat{p}_{j_{k}}^{2}}{p_{j_{1}}^{2}...p_{j_{k}}^{2}}\right) }{\sum_{(j_{1},..,j_{k})\in I_{k}^{+}} D_{j_{1},..,j_{k}}\prod_{l=1}^{k}(\frac{\lambda_{1}}{\lambda_{j_{l}}}-1)}\right \rvert\leq \left( 1+C\sqrt{\frac{\log n}{nL}}\right)^{2k}=1+O\left( 2k\sqrt{\frac{\log n}{nL}}\right) .
\end{equation}
Let $i=n$. Then for all $(j_{1},..,j_{k})\in I_{k,n}^{+}$ we have 
$\prod_{l=1}^{k}(\frac{\lambda_{1}}{\lambda_{j_{l}}}-1)>0 $ if $k$ is even and $\prod_{l=1}^{k}(\frac{\lambda_{1}}{\lambda_{j_{l}}}-1)<0 $ if $k$ is odd. Thus on the event $\mathcal{A}$ we have
\begin{equation}
\label{eq:first-product-2}
 \left \lvert\dfrac{\sum_{(j_{1},..,j_{k})\in I_{k}^{+}} \left( D_{j_{1},..,j_{k}}\prod_{l=1}^{k}(\frac{\lambda_{n}}{\lambda_{j_{l}}}-1)\frac{\hat{p}_{j_{1}}^{2}...\hat{p}_{j_{k}}^{2}}{p_{j_{1}}^{2}...p_{j_{k}}^{2}}\right) }{\sum_{(j_{1},..,j_{k})\in I_{k}^{+}} D_{j_{1},..,j_{k}}\prod_{l=1}^{k}(\frac{\lambda_{1}}{\lambda_{j_{l}}}-1)} \right \rvert\leq \left( 1+C\sqrt{\frac{\log n}{nL}}\right)^{2k}=1+O\left( 2k\sqrt{\frac{\log n}{nL}}\right) .
 \end{equation}
 
To conclude from (\ref{eq:first product-1}), (\ref{eq:first-product-2}) and (\ref{eq:second product}) we have for $i=1$ and $i=n$
$$\left \lvert \dfrac{\hat{Q}_{k}(\lambda_{i})}{Q^{*}_{k}(\lambda_{i})}\right \rvert\leq1+O\left( 2k\sqrt{\frac{\log n}{nL}}\right). $$
\end{rmq}
\end{proof}

\bibliographystyle{apalike}
\bibliography{PLSbiblio}

\begin{thebibliography}{}

\bibitem[Blanchard and Math{\'e}, 2012]{BLAN12}
Blanchard, G. and Math{\'e}, P. (2012).
\newblock Discrepancy principle for statistical inverse problems with
  application to conjugate gradient iteration.
\newblock {\em Inverse Problems}, 28(11):115011.

\bibitem[Boulesteix and Strimmer, 2007]{BOU07}
Boulesteix, A.-L. and Strimmer, K. (2007).
\newblock Partial least squares: a versatile tool for the analysis of
  high-dimensional genomic data.
\newblock {\em Briefings in bioinformatics}, 8(1):32--44.

\bibitem[Butler and Denham, 2000]{BUTLER00}
Butler, N.~A. and Denham, M.~C. (2000).
\newblock The peculiar shrinkage properties of partial least squares
  regression.
\newblock {\em Journal of the Royal Statistical Society: Series B (Statistical
  Methodology)}, 62(3):585--593.

\bibitem[De~Jong, 1995]{JONG95}
De~Jong, S. (1995).
\newblock Pls shrinks.
\newblock {\em Journal of chemometrics}, 9(4):323--326.

\bibitem[Delaigle and Hall, 2012]{DEL12}
Delaigle, A. and Hall, P. (2012).
\newblock Methodology and theory for partial least squares applied to
  functional data.
\newblock {\em The Annals of Statistics}, 40(1):322--352.

\bibitem[Engl et~al., 1996]{MR1408680}
Engl, H.~W., Hanke, M., and Neubauer, A. (1996).
\newblock {\em Regularization of inverse problems}, volume 375 of {\em
  Mathematics and its Applications}.
\newblock Kluwer Academic Publishers Group, Dordrecht.

\bibitem[Frank and Friedman, 1993]{FRANCK93}
Frank, l.~E. and Friedman, J.~H. (1993).
\newblock A statistical view of some chemometrics regression tools.
\newblock {\em Technometrics}, 35(2):109--135.

\bibitem[Garthwaite, 1994]{GAR94}
Garthwaite, P.~H. (1994).
\newblock An interpretation of partial least squares.
\newblock {\em Journal of the American Statistical Association},
  89(425):122--127.

\bibitem[Goutis, 1996]{GOU96}
Goutis, C. (1996).
\newblock Partial least squares algorithm yields shrinkage estimators.
\newblock {\em The Annals of Statistics}, 24(2):816--824.

\bibitem[Helland, 1988]{HE88}
Helland, I.~S. (1988).
\newblock On the structure of partial least squares regression.
\newblock {\em Communications in statistics-Simulation and Computation},
  17(2):581--607.

\bibitem[Helland, 1990]{HE90}
Helland, I.~S. (1990).
\newblock Partial least squares regression and statistical models.
\newblock {\em Scandinavian Journal of Statistics}, pages 97--114.

\bibitem[Helland, 2001]{HE01}
Helland, I.~S. (2001).
\newblock Some theoretical aspects of partial least squares regression.
\newblock {\em Chemometrics and Intelligent Laboratory Systems}, 58(2):97--107.

\bibitem[Jolliffe, 1982]{JO82}
Jolliffe, I.~T. (1982).
\newblock A note on the use of principal components in regression.
\newblock {\em Applied Statistics}, pages 300--303.

\bibitem[Kr{\"a}mer, 2007]{KRA07}
Kr{\"a}mer, N. (2007).
\newblock An overview on the shrinkage properties of partial least squares
  regression.
\newblock {\em Computational Statistics}, 22(2):249--273.

\bibitem[L{\^e}~Cao et~al., 2008]{MR2457048}
L{\^e}~Cao, K.-A., Rossouw, D., Robert-Grani{\'e}, C., and Besse, P. (2008).
\newblock A sparse {PLS} for variable selection when integrating omics data.
\newblock {\em Stat. Appl. Genet. Mol. Biol.}, 7(1):Art. 35, 31.

\bibitem[Lingjaerde and Christophersen, 2000]{LIN00}
Lingjaerde, O.~C. and Christophersen, N. (2000).
\newblock Shrinkage structure of partial least squares.
\newblock {\em Scandinavian Journal of Statistics}, 27(3):459--473.

\bibitem[Martens and Naes, 1992]{MA92}
Martens, H. and Naes, T. (1992).
\newblock {\em Multivariate calibration}.
\newblock Wiley.

\bibitem[Naes and Martens, 1985]{NAES85}
Naes, T. and Martens, H. (1985).
\newblock Comparison of prediction methods for multicollinear data.
\newblock {\em Communications in Statistics-Simulation and Computation},
  14(3):545--576.

\bibitem[Phatak and de~Hoog, 2002]{PHATAK02}
Phatak, A. and de~Hoog, F. (2002).
\newblock Exploiting the connection between pls, lanczos methods and conjugate
  gradients: alternative proofs of some properties of pls.
\newblock {\em Journal of Chemometrics}, 16(7):361--367.

\bibitem[Saad, 1992]{SAAD92}
Saad, Y. (1992).
\newblock {\em Numerical methods for large eigenvalue problems}, volume 158.
\newblock SIAM.

\bibitem[Wold, 1985]{WO85}
Wold, H. (1985).
\newblock Partial least squares.
\newblock {\em Encyclopedia of statistical sciences}.

\bibitem[Wold et~al., 1983]{WO83}
Wold, S., Martens, H., and Wold, H. (1983).
\newblock The multivariate calibration problem in chemistry solved by the pls
  method.
\newblock In {\em Matrix pencils}, pages 286--293. Springer.

\bibitem[Wold et~al., 2001]{WOLD01PLS}
Wold, S., Sj{\"o}str{\"o}m, M., and Eriksson, L. (2001).
\newblock Pls-regression: a basic tool of chemometrics.
\newblock {\em Chemometrics and intelligent laboratory systems},
  58(2):109--130.

\end{thebibliography}

\end{document}